\documentclass[3p]{elsarticle}

\journal{}

\usepackage{graphicx}

\usepackage{amsmath, amssymb, amsthm}
\usepackage{graphicx}

\usepackage[colorinlistoftodos]{todonotes}
\usepackage[colorlinks=true, allcolors=blue]{hyperref}
\usepackage{mathtools}

\usepackage{tikz}
\usepackage{tikz-cd}

\tikzcdset{graphstyle/.append style={row sep=1.5em, column sep=1.5em, nodes={inner sep=0.5pt}, arrows={-stealth}}}

\usetikzlibrary{calc, arrows.meta, bending}
\usepackage[all]{xy}

\usepackage{multicol}

\usepackage{booktabs, caption, makecell}
\usepackage{threeparttable}
\usepackage{placeins}
\usepackage{diagbox}
\usepackage{multirow}

\usepackage{algorithm}
\usepackage{algpseudocode}

\makeatletter
\newcommand{\StatexIndent}[1][3]{%
  \setlength\@tempdima{\algorithmicindent}%
  \Statex\hskip\dimexpr#1\@tempdima\relax}
\makeatother

\usepackage{enumitem}
\setlist[enumerate]{labelindent = 0pt,leftmargin=*,label=$(\arabic*)$}
\setlist[itemize]{label=$\bullet$}
\numberwithin{equation}{section}

\newtheorem{theorem}{Theorem}[section]
\newtheorem{lemma}[theorem]{Lemma}
\newtheorem{corollary}[theorem]{Corollary}
\newtheorem{proposition}[theorem]{Proposition}

\theoremstyle{definition}
\newtheorem{definition}[theorem]{Definition}
\newtheorem{example}[theorem]{Example}

\newtheorem{remark}[theorem]{Remark}

\DeclareMathOperator{\rep}{rep}
\DeclareMathOperator{\soc}{soc}

\DeclareMathOperator{\Hom}{Hom}

\DeclareMathOperator{\supp}{supp}

\DeclareMathOperator{\rank}{rank}
\DeclareMathOperator{\Cov}{Cov}
\DeclareMathOperator{\pr}{pr}
\DeclareMathOperator{\Gr}{Gr}

\DeclareMathOperator{\Seg}{Seg}

\newcommand{\vect}{\mathrm{vect}}
\newcommand{\op}{^\mathrm{op}}

\usepackage[e]{esvect}
\newcommand{\Af}[1]{\vv*{A}{#1}}
\newcommand{\Gf}[1]{\vv*{G}{#1}}

\newcommand{\II}[1]{\mathbb{I}_{#1}}

\newcommand{\intv}{\mathbb{I}}

\newcommand{\br}[1]{\left[\!\left[#1\right]\!\right]}

\newcommand{\ang}[1]{\langle #1 \rangle}

\newcommand{\calC}{\mathcal{C}}
\newcommand{\calL}{\mathcal{L}}
\newcommand{\bbZ}{\mathbb{Z}}
\newcommand{\bbZnn}{\bbZ_{\ge 0}}
\newcommand{\bbN}{\mathbb{N}}

\newcommand{\suchthat}{\ifnum\currentgrouptype=16 \;\middle|\;\else\mid\fi}

\newcommand{\udim}{\underline{\operatorname{dim}}}

\newcommand{\nd}[2]{d_{#1} (#2)}
\newcommand{\QComp}[2]{{#2}^{#1}}
\newcommand{\VComp}[3]{\operatorname{Comp}^{#1}_{#2}(#3)}

\newcommand{\dbarfun}[2]{\underline{d}^{#1}_{#2}}
\newcommand{\dbar}[3]{\dbarfun{#1}{#2}(#3)}

\newcommand{\rss}{{\textrm{ss}}}
\newcommand{\rcc}{{\textrm{cc}}}
\newcommand{\tot}{{\textrm{tot}}}
\newcommand{\aprx}[2]{\delta^{#1}_{#2}} 
\newcommand{\aprxM}[2]{\delta^{#1}(#2)}

\newcommand{\blank}{\operatorname{-}}
\newcommand{\id}{1\kern-.25em{\text{{\rm l}}}} 
\newcommand{\smat}[1]{\begin{smallmatrix}#1\end{smallmatrix}}

\usepackage{pifont}
\newcommand{\cmark}{\ding{51}}%
\newcommand{\xmark}{\ding{55}}%

\newcommand{\Mat}{\operatorname{Mat}}

\newcommand{\EDIT}[1]{{#1}}
\begin{document}

\begin{frontmatter}
  \title{On Approximation of $2$D Persistence Modules by Interval-decomposables\tnoteref{mytitlenote}}

  \tnotetext[mytitlenote]{
    This work is partially supported by Japan Science and Technology Agency CREST Mathematics (15656429).
    H.A.\ is supported by JSPS Grant-in-Aid for Scientific Research (C) 18K03207, and
    JSPS Grant-in-Aid for Transformative Research Areas (A) (22A201).
    E.G.E.\ is supported by JSPS Grant-in-Aid for Transformative Research Areas (A) (22H05105).
    K.N.\ is supported by JSPS Grant-in-Aid for Transformative Research Areas (A) (20H05884).
    M.Y.\ is supported by JSPS Grant-in-Aid for Scientific Research (C) (20K03760).
    H.A.\ and M.Y.\ are partially supported by Osaka Central Advanced Mathematical Institute (MEXT Joint Usage/Research Center on Mathematics and Theoretical Physics JPMXP0619217849 \EDIT{and MEXT Promotion of Distinctive Joint Research Center Program JPMXP0723833165}).
    A part of this work was performed while E.G.E., K.N., and M.Y.\ were affiliated with
    Center for Advanced Intelligence Project, RIKEN, Tokyo, Japan.
  }

  \author[shizu,kuias,ocami]{Hideto Asashiba}
  \ead{asashiba.hideto@shizuoka.ac.jp}

  \author[kobe]{Emerson G. Escolar\corref{mycorrespondingauthor}}
  \cortext[mycorrespondingauthor]{Corresponding author}
  \ead{e.g.escolar@people.kobe-u.ac.jp}

  \author[shimane]{Ken Nakashima}
  \ead{knakashima@mat.shimane-u.ac.jp}

  \author[ocami]{Michio Yoshiwaki}
  \ead{yosiwaki@sci.osaka-cu.ac.jp, michio.yoshiwaki@omu.ac.jp}

  \address[shizu]{Faculty of Science, Shizuoka University, Japan}
  \address[kuias]{Institute for Advanced Study, Kyoto University, Japan}
  \address[ocami]{Osaka Central Advanced Mathematical Institute, Osaka, Japan}
  \address[kobe]{Graduate School of Human Development and Environment, Kobe University, Japan}
  \address[shimane]{Faculty of Materials for Energy, Shimane University, Japan}

  \begin{abstract}
    In this work, we propose a new invariant for $2$D persistence modules called the compressed multiplicity and show that it generalizes the notions of the dimension vector and the rank invariant.
    In addition, for a $2$D persistence module $M$,
    we propose an ``interval-decomposable \EDIT{replacement}'' $\aprxM{\ast}{M}$
    (in the split Grothendieck group of the category of persistence modules),
    which
    is expressed
    by a pair of interval-decomposable modules,
    that is, its positive and negative parts.
    We show that $M$ is interval-decomposable if and only if 
    $\aprxM{\ast}{M}$ is equal to $M$ in the split Grothendieck group.
    Furthermore, even for modules $M$ not necessarily interval-decomposable, $\aprxM{\ast}{M}$ preserves the dimension vector and the rank invariant of $M$.
    In addition, we provide an algorithm to compute $\aprxM{\ast}{M}$ (a high-level algorithm in the general case, and a detailed algorithm for the size $2\times n$ case).
  \end{abstract}

  \begin{keyword}
    Representation theory \sep Multidimensional persistence \sep Intervals
    \MSC[2020] 16G20 \sep 55N31
  \end{keyword}
\end{frontmatter}



\section{Introduction}

Persistent homology \cite{edelsbrunner2002topological,edelsbrunner2008persistent} is one of the main tools in the rapidly growing field of topological data analysis. Given a filtration -- a one-parameter increasing sequence of spaces -- persistent homology captures the persistence of topological features such as connected components, holes, voids, etc. in the filtration.
Here, the persistence of features is quantified by birth and death parameter values. This can be summarized compactly by the so-called persistence diagram, which is the multiset of birth-death pairs drawn on the plane with multiplicity.

Algebraically, the persistence diagram can be explained
as resulting from a structure theorem (the Krull-Schmidt theorem (Theorem~\ref{thm:KS}) and Gabriel's Theorem~\cite{gabriel1972unzerlegbare}) of persistence modules, which can also be regarded as representations of certain quivers.
\EDIT{We however note that early definitions of persistence diagrams and related ideas used an inclusion-exclusion formula instead of this algebraic point of view. See for example \cite{landi1997new}, \cite{frosini1999size}, \cite{robins1999towards}, \cite{cohen2005stability}, \cite{cerri2016hausdorff}  
  and others\footnote{\EDIT{We do not provide a historical review here.
    Since we highlight in the next paragraph
    the algebraic difficulties of its generalization, multidimensional persistence,
    we have adopted this algebraic explanation here.}}.}
See Section~\ref{sec:background} for detailed definitions.

One way to deal with multiparametric data is to use multidimensional persistence \cite{carlsson2009theory}. However, multidimensional persistence presents theoretical difficulties that hinder the construction of a persistence diagram as in one-dimensional persistence.
\EDIT{In this work, for simplicity,
  we restrict our attention to the two-parameter case, and consider $2$D persistence modules.
  We thus study representations of the
  equioriented $m\times n$ commutative grid $\Gf{m,n}$ (a finite portion of the $2$D grid $\mathbb{Z}^2$).
  Even the $2$D case is sufficiently difficult.}
In particular, there is no complete discrete invariant that captures all isomorphism classes of indecomposable persistence modules \cite{carlsson2009theory}.
Another way of expressing this difficulty is that 
the grid
$\Gf{m,n}$ of sufficiently large size ($m,n \geq 2$ and $mn \geq 12$, see \cite[Theorem~1.3]{bauer2020cotorsion}, \cite[Theorem~2.5]{leszczynski1994representation}, \cite[Theorem~5]{leszczynski2000tame}) is of wild representation type.

One way to avoid this problem is to consider only a restricted class of persistence modules. Inspired by $1$D persistence, there has been much interest in the so-called interval-decomposable representations, which are direct sums of interval representations (Definition~\ref{def:interval_rep}). The work \cite{asashiba2022interval} studied this family of representations and provided a criterion to determine whether or not a given persistence module is interval-decomposable.

It is hoped that most persistence modules coming from ``real-world data'' contain very few or indeed no non-interval summands. Let us consider the silica glass example computed in \cite{escolar2016persistence}, which  compares the atomic configuration of silica glass with its configuration after physical pressurization. The underlying bound quiver is the commutative ladder $C\!L_3(fb)$, with only two non-interval indecomposable representations given by dimension vectors $\left(\smat{1 & 1 & 1 \\ 1 & 2 & 1 }\right)$ and  $\left(\smat{1 & 2 & 1 \\ 0 & 1 & 0}\right)$. Then, the numerical result in \cite{escolar2016persistence} has $\left(\smat{1 & 1 & 1 \\ 1 & 2 & 1 }\right)$ appearing with only multiplicity $1$ and $\left(\smat{1 & 2 & 1 \\ 0 & 1 & 0}\right)$ with multiplicity $0$, in an example with more than ten thousand indecomposable summands. While in the slightly different setting of a non-equioriented commutative ladder, this provides an example where the non-interval part is minute compared to the interval-decomposable part.

On the other hand, the work \cite{buchet_et_al:socg} argues via a geometric example that the non-interval indecomposables may contain important information that should not be ignored, and that even in relatively simple geometric point clouds embedded in $\mathbb{R}^3$, indecomposable summands with arbitrarily large dimension (as a vector space) may be present. These large indecomposable summands are clearly not intervals.

In this work, we take neither position, but instead propose a method to replace an arbitrary persistence module $M\in\rep{\Gf{m,n}}$ by an object $\aprxM{\ast}{M}$ in the split Grothendieck group that is interval-decomposable.
The \EDIT{\emph{interval-decomposable replacement} (or \emph{interval-decomposable approximation})}\footnote{\EDIT{In an earlier version of this work we called $\aprxM{\ast}{M}$ an ``interval-decomposable approximation'' of $M$. However, more recent works such as \cite{blanchette2021homological,asashiba2023approximation} have used the term ``approximation'' in a relative homological sense. Thus, here, we have opted to mainly use the term ``replacement'' to avoid confusion, and also because it more closely describes how we think of $\aprxM{\ast}{M}$ relative to $M$. In the end of subsection~\ref{subsec:mobiustda} we provide some references to the point of view from relative homological algebra.}\label{footnotereplacement}} $\aprxM{\ast}{M}$  (Definition~\ref{def:tildeM}) is expressed by a pair of interval-decomposable modules, that is, its positive part $\aprxM{\ast}{M}_+$ and negative part $\aprxM{\ast}{M}_-$ (see \eqref{eq:pos-neg}).

To construct $\aprxM{\ast}{M}$, we first define what we call the \emph{compressed multiplicity} (Definition~\ref{defn:compressedmultiplicities}) of $M$ by a compression operation that picks up information in $M$ restricted to certain essential vertices of intervals.

The intuition behind the compressed multiplicity can be explained as follows.
As an initial goal, we want to compute the multiplicity of an interval $I$ as a direct summand of $M$.
Indeed, the work \cite{asashiba2022interval} presents an algorithm for this computation.
However, as this may not be straightforward, in this work we adopt a different approach. We first compress both $M$ and $I$ by restricting the underlying domain to certain essential vertices of $I$, and compute the multiplicity in the representation category with smaller underlying domain.

In the equioriented commutative ladder \cite{escolar2016persistence} case ($\Gf{2,n}$), the compression operation reduces the underlying bound quiver to a representation-finite bound quiver.  This enables easy computation of the  compressed multiplicity using preexisting algorithms.

We show that the compressed multiplicity in fact generalizes the notions of dimension vector (Proposition~\ref{prop:dimvec}) and rank invariant (Proposition~\ref{prop:rankinv}). Furthermore, we exhibit representations that can be distinguished by their compressed multiplicities but not by their rank invariants. We thus propose the compressed multiplicity as a new, finer invariant for $2$D persistence modules. Moreover, we show that for interval-decomposable representations, the multiplicity can be recovered from the compressed multiplicity (Theorem~\ref{thm:interval}).

Then, the object $\aprxM{\ast}{M}$ is defined using the M\"obius inversion of the compressed multiplicity of $M$.
This is a generalization of the well-known fact that the multiplicities of interval summands in $1$D persistence modules can be obtained via an application of inclusion-exclusion on the ranks of the linear maps. 
In fact, early works around ideas related to persistence diagrams used this as the definition. See for example \cite{landi1997new}, \cite{frosini1999size}, \cite{robins1999towards}, \cite{cohen2005stability}, \cite{edelsbrunner2008persistent}, 
\cite{chazal2009proximity}, and others.

That is, the persistence diagram is simply the M\"obius inversion of the rank invariant.
We note that several works have already exploited this observation to define ``generalized persistence diagrams'' in general settings. In Subsection~\ref{subsec:mobiustda}, we review some of them
and contrast them with our work.

In the case that $M$ is interval-decomposable, it follows that $\aprxM{\ast}{M}$
is equal to $M$ viewed as an element $\br{M}$ of the split Grothendieck group (Theorem~\ref{thm:intervaltilde});
that is, $\aprxM{\ast}{M}_+ \cong M$ and $\aprxM{\ast}{M}_- = 0$.
Furthermore, we show that even for modules $M$ not necessarily interval-decomposable, $\aprxM{\ast}{M}$ preserves the dimension vector and the rank invariant of $M$ (Corollary~\ref{cor:dimvec}, Theorem~\ref{thm:rank}). In this sense, we think of $\aprxM{\ast}{M}$ as an interval-decomposable \EDIT{``replacement'' for} $M$.

We organize this work as follows. In Section~\ref{sec:background}, we review the necessary background from representation theory of bound quivers and poset theory\footnote{In the persistence literature, there are at least two ways to consider multidimensional persistence modules: as representations of certain posets, or as representations of certain bound quivers. These are equivalent for the cases we are interested in. See Subsection~\ref{subsec:incidence} for a more detailed discussion. For example, \cite{kim2018generalized} uses the point of view of poset representations. In this work, we consider persistence modules as representations of bound quivers, except when comparing with the literature that uses poset representations. We reserve our use of posets for the poset whose elements happen to be interval representations, and do not directly consider representations of posets.}, and then, in Section~\ref{sec:interval_lattice}, we study the poset of interval representations.
In Section~\ref{sec:comp}, we introduce our concept of compressed multiplicities and study its properties.
In Section~\ref{sec:approximation}, we give the construction of $\aprxM{\ast}{M}$ from $M$ via M\"obius inversion of the compressed multiplicity and give some results about its properties.
In Section~\ref{sec:algorithm}, we discuss the computation of our proposed compressed multiplicity and the interval-decomposable \EDIT{replacement}.

\subsection{\EDIT{Related literature}}
\label{subsec:mobiustda}

The prior work of Patel~\cite{patel2018generalized} used the idea of M\"obius inversion in order to define generalized persistence diagrams, but only in the setting of persistence modules over $(\mathbb{R},\leq)$ \cite[Definition~2.1]{patel2018generalized}.
  Then, the work \cite{mccleary2022edit} defines a concept of a persistence diagram
  for a filtered simplicial complex over any finite metric lattice $P$, by using a M\"obius inversion.
  Furthermore, they show a stability theorem with respect to the edit distance for filtrations of a fixed simplicial complex, and the bottleneck distance for persistence diagrams.
  In that work, the domain of the persistence diagrams is $\bar{P}$,
  which is the set of what we call the segments $[x,y] = \{z \in P \mid x \leq z \leq y\}$ of $P$.
  Applied to the commutative grid $P = \Gf{m,n}$ (viewed as a poset),
  we get a persistence diagram descriptor over the rectangles in $\Gf{m,n}$,
  different from our descriptor over what we call intervals (Definition~\ref{def:intervalsubq}) of $\Gf{m,n}$.
  We also note that the poset structure for $\bar{P}$ considered in \cite{mccleary2022edit} is different from
  the poset structure we give the set of intervals of $\Gf{m,n}$.

  Our work can be compared with the following prior work of Kim and Memoli~\cite{kim2018generalized}, 
  which we were made aware of
  by a reviewer after an initial version of this work was sent for review.
  %
  In Table~\ref{table:comparison}, we provide a rough overview of the different settings and a correspondence of some of the results, which we explain in detail below.

\begin{table}[!h]\centering \caption{Settings and Some Similar Results${}^{\ast}$}
  \begin{threeparttable}
    \begin{tabular}{lcc}
      \toprule\midrule
      & \thead{This work} & \thead{Kim and Memoli~\cite{kim2018generalized}} \\\midrule\midrule
      (1) Underlying setting & commutative grid $\Gf{m,n}$ & poset\tnote{1}{ } $P$ \\ \midrule
      (2) Target category & $\vect_K$ & category\tnote{2}{ } $\mathcal{C}$ \\ \midrule
      (3) Domain of invariant & $\II{m,n}$ \tnote{3} & $\mathbf{Con}(P)$\tnote{4} \\ \midrule
      (4) Invariant proposed & \scriptsize \makecell{compressed multiplicities\\ $\dbarfun{\ast}{M}:\II{m,n}\rightarrow \mathbb{N}$} & \scriptsize \makecell{generalized rank invariant\tnote{5}{ } \\ $\mathrm{rk}(M):\mathbf{Con}(P) \rightarrow \mathcal{J}(\mathcal{C})$} \\ \midrule
      (5) Inversion & \makecell{$\aprx{\ast}{M} : \II{m,n} \rightarrow \mathbb{Z}$ } & \scriptsize\makecell{generalized persistence diagram \\ $\mathrm{dgm}^P(M) :  \mathbf{Con}(P) \rightarrow \Gr(\mathcal{C})$} \\\midrule
      (6) Object & \scriptsize \makecell{interval-decomposable\\ \EDIT{replacement} \\ $\aprxM{\ast}{M} \in \Gr(\rep\Gf{m,n})$} & \textemdash\tnote{6} \\
      \midrule\midrule
      (7) \scriptsize\makecell{{ from proposed invariant}\\{ to true multiplicities}\\{ (interval-decomposable)}} & Theorem~\ref{thm:interval} & \cite[Theorem~3.14]{kim2018generalized}\\
      \midrule
      (8) \scriptsize\makecell{{ from true multiplicities}\\{ to proposed invariant}\\{ (interval-decomposable)}} & Lemma~\ref{lem:key} & \cite[Proposition~3.17]{kim2018generalized}\\
      \midrule
      (9)\scriptsize \makecell{{ Interpretation as}\\{ M\"obius inversion}} & Theorem~\ref{thm:mobiusequal} & \cite[Proposition~3.19]{kim2018generalized}\\
      \bottomrule\addlinespace[1ex]
    \end{tabular}
    \begin{tablenotes}\footnotesize
    \item[*] \textbf{This table is not intended to be a comprehensive summary of all results.}\tnote{7}
    \item[1] Essentially finite connected poset
    \item[2] Essentially small, symmetric monoidal category satisfying \cite[Convention~2.3]{kim2018generalized}
    \item[3] interval (connected and convex) subquivers
    \item[4] path-connected subposets
    \item[5] See \cite[Definition~3.5]{kim2018generalized}. The codomain $\mathcal{J}(\mathcal{C})$ is the set of isomorphism classes of $\mathcal{C}$.
    \item[6] Not explicitly defined. See however, \cite[Remark~3.22]{kim2018generalized}.
    \item[7] For example, \cite{kim2018generalized} contains results concerning Reeb graphs, which can be viewed as functors from the ``zigzag poset'' to the category of finite sets.
    \end{tablenotes}
  \end{threeparttable}
  \label{table:comparison}
\end{table}

While Kim and Memoli~\cite{kim2018generalized} consider a very general setting, we restrict our attention to $K$-representations of the commutative grid $\Gf{m,n}$ (see rows (1) and (2) of Table~\ref{table:comparison}). Since $\Gf{m,n}$ can be viewed as a poset $P$, which happens to be essentially finite and connected, their setting contains ours.
\EDIT{First,} 
the domains of the proposed invariants (see row (3) of Table~\ref{table:comparison}) are different.  We note that $\mathbf{Con}(P)$, the set of all path-connected subposets is in general different from the set of all interval subposets, and this is indeed the case for $P=\Gf{m,n}$.
The set $\mathbf{Con}(P)$ contains subposets that cannot be realized as the support of some persistence module.
For example, viewing $\Gf{2,2}$ as a poset with Hasse diagram (both filled and unfilled circles):
\begin{equation}
  \label{diag:noninterval}
  \begin{tikzcd}[graphstyle]
    \circ \rar & \bullet\\
    \bullet \rar\uar & \bullet \uar
  \end{tikzcd},
\end{equation}
the subposet $C$ given by the filled-in circles is in $\mathbf{Con}(P)$.
However, \EDIT{this is not an interval (Definition~\ref{def:intervalsubq}),}
and there is no thin\footnote{A persistence module is said to be \emph{thin} if all of its vector spaces have dimension at most $1$. For example, interval persistence modules are thin.} indecomposable persistence module over $\Gf{2,2}$ with support given by $C$, as a commutativity relation will be violated otherwise.
\EDIT{We do note however that subsequent works on the generalized rank invariant \cite{dey2021computing,kim2021generalized} have restricted the domain to the set of all intervals, instead of using $\mathbf{Con}(P)$.}

Furthermore, the proposed invariants (row (4) of Table~\ref{table:comparison}) are different.
We first note that both papers use of the idea of restricting the input persistence module $M$ to define the respective invariants. In \cite{kim2018generalized}, $M$ is restricted to $I \in \mathbf{Con}(P)$ to obtain $M|_I$. In the case that $I$ is in fact an interval, this corresponds to applying what we call the ``total compression'' functor (Definition~\ref{defn:compressionfunctor})
in a more general setting. 

Kim and Memoli~\cite{kim2018generalized} then
defines the value of their generalized rank invariant at $I\in \mathbf{Con}(P)$ to be ``the isomorphism class of the image of the canonical limit-to-colimit map'' for $M_I$. Of course, in the case that the target category $\mathcal{C}$ is $\vect_K$, the category of finite-dimensional $K$-vector spaces, this value can be fully characterized by the dimension of the image.
In fact, one version of our invariant, which we call the ``total compressed multiplicity'', coincides with the dimensions of their generalized rank invariant (see Remark~\ref{remark:totkm}).
\begin{remark}
However, we emphasize that 
this total compressed multiplicity is not the main emphasis of this work.
Instead, 
we propose the use of the source-sink ($\rss$-)compression yielding smaller representations (compared to $M|_I$), by further restriction to what we call the essential vertices of $I$.
We note that these do not coincide with the generalized rank invariant of \cite{kim2018generalized} for fixed $I$.
  See Example~\ref{example:unequal}. However, if we allow to change the form of the ``input'' to generalized rank invariant and broaden its domain of definition, we indeed recover values of our source-sink multiplicity (See Remark~\ref{remark:sskm}).
\end{remark}


\EDIT{
  Our interval-decomposable replacement $\aprxM{\ast}{M}$
  (and the generalized persistence diagram $\mathrm{dgm}^P(M)$ of \cite{kim2018generalized})
  are signed invariants, with positive and negative parts.
  Many recent works, such as \cite{botnan2021signed,blanchette2021homological,chacholski2023koszul,asashiba2023relative,asashiba2023approximation,blanchette2023exact} and others, are further studying signed invariants for persistence, especially from the point of view of relative homological algebra.
  While our use of the split Grothendieck group to formalize
  the interval-decomposable replacement $\aprxM{\ast}{M}$  (Definition~\ref{def:tildeM})
  is related to the (relative) homological algebra point of view,
  a full treatment is beyond the scope of this manuscript.
  Thus, we only provide brief comments below and
  point the interested reader to the literature listed above.
  For example, \cite{botnan2021signed} studied rank decompositions of the rank invariant
  and their connections to the generalized persistence diagrams and resolutions
  relative to a so-called rank-exact structure.
  \cite{blanchette2021homological} develops a framework for ``homological invariants''
  and studies several existing invariants in this framework.
  The work \cite{chacholski2023koszul} (and \cite{asashiba2023relative})
  provides a method for computing relative Betti numbers.
  In the case of resolutions relative to interval-decomposables, these
  are simply the multiplicities of each interval appearing in each term of a minimal resolution.
  One of the results of our subsequent work \cite{asashiba2023approximation}
  is a connection between (a modified version of) compressed multiplicity
  and resolutions relative to the interval-decomposables (and thus homological invariants).
  \cite{blanchette2023exact} is mostly expository,
  but contains some novel results:
  explicit descriptions of irreducible morphisms between relative projectives,
  and a way to ``lift'' the theory to certain infinite posets.
}
\FloatBarrier



\section{Background}
\label{sec:background}

\subsection{Representation Theory}

We first recall some fundamental terminologies of representations of quivers
(see \cite{assem2006elements} for instance\footnote{
Note that there is a difference between our convention and theirs
in the order of arrows in paths. Namely, the path $\alpha_n \cdots \alpha_1$ in this paper
is written as $\alpha_1 \cdots \alpha_n$ in their book}).

A \emph{quiver} $Q$ is a quadruple $(Q_0,Q_1,s,t)$ of sets $Q_0, Q_1$ of vertices and arrows, respectively and maps $s,t:Q_1 \to Q_0$ that give the source and target vertices, respectively, of the arrows. We denote an arrow $\alpha$ with source $s(\alpha) =x $ and target $t(\alpha)=y$ by $\alpha:x\rightarrow y$.
The \emph{opposite} quiver $Q\op$ of $Q$ is the quiver given by $(Q_0, Q_1, t,s)$, namely the quiver obtained from $Q$ by reversing all arrows.
In this paper, all quivers $Q$ are assumed be {\em finite}, namely, $Q_0$ and $Q_1$ are finite.

Throughout this work, we fix a field $K$. Let $Q$ be a quiver. 
A {\em representation} $V$ of $Q$ (over $K$) is a family $(V(x), V(\alpha))$ of a vector space $V(x)$ for each vertex $x\in Q_0$ and a linear map $V(\alpha):V(x) \rightarrow V(y)$ for each arrow $\alpha:x\rightarrow y$ in $Q_1$.

The \emph{dimension vector} $\udim (V)$ of a representation $V$ of $Q$ is defined as the tuple
\[
  \udim (V) := (\dim V(x) )_{x \in Q_0}.
\]
It is customary to display the dimension vector by writing each number $\dim V(x)$ relative to where the vertex $x$ is located on an illustration of the quiver $Q$.
\EDIT{While the dimension vector does not uniquely specify the representation $V$,
  if it is clear from context, we also use the dimension vector to stand for $V$.}
The \emph{dimension} of $V$ is $\dim V := \sum\limits_{x\in Q_0} \dim V(x)$.
A representation $V$ of $Q$ is said to be {\em finite-dimensional} if $\dim V < \infty$.
In this work, by representation we mean finite-dimensional representation.

Let $V$ and $W$ be representations of $Q$. A morphism $f : V \rightarrow W$ from $V$ to $W$ is a family $(f_x)_{x\in Q_0}$ of linear maps $f_x : V(x)\rightarrow W(x)$ such that the following diagram commutes
for each arrow $\alpha : x \to y$:
\[
  \begin{tikzcd}
    V(x) \arrow[r, "f_x"] \arrow[d, "V(\alpha)"'] & W(x) \arrow[d, "W(\alpha)"] \\
    V(y) \arrow[r, "f_y"'] & W(y).
  \end{tikzcd}
\]
The composition of morphisms $f = (f_x)_{x\in Q_0} :V\rightarrow W$ and $g  = (g_x)_{x\in Q_0} :U \rightarrow V$ is defined in the obvious way: $f\circ g  :U \rightarrow W$ is given by $(f\circ g)_x = f_x \circ g_x$.
We denote by $\rep Q$ the category of finite-dimensional representations of $Q$ together with these morphisms and this composition.

For each vertex $i\in Q_0$, we have the {\em path of length $0$ at $i$}, which is denoted by $e_i$. 
For a given positive integer $n$, 
a {\em path $\mu$ of length} $n$ is a sequence $\alpha_n \cdots \alpha_1$ of arrows $\alpha_i$ such that $t(\alpha_i) = s(\alpha_{i+1})$ for all $i=1,\cdots,n-1$. The source vertex of $\mu$ is $s(\alpha_1)$, while its target vertex is $t(\alpha_n)$.
An $m$-tuple $\mu_1,\cdots,\mu_m$ of paths is said to be {\em parallel} if they all have the same source vertex and the same target vertex. 
A \emph{relation} $\rho$ in $Q$ is a formal sum $\rho=\sum\limits_{i=1}^{m} t_i \mu_i$ of parallel paths $\mu_i$,
\EDIT{where each path $\mu_i$ is of length at least $2$ and each $t_i$ is in $K$}.
A pair $(Q,R)$ of a quiver $Q$ and a set $R$ of relations is called a {\em bound quiver}.

A relation $\rho$ is called a {\em commutativity relation} if $\rho=\mu_1 - \mu_2$ for some two parallel paths $\mu_1,\mu_2$.
If $R$ is the set of all possible commutativity relations in $Q$, $(Q,R)$ is called {\em a quiver with full commutativity relations}.

Let $(Q,R)$ be a bound quiver and let $V$ be a representation of $Q$.
Put $V(\mu):= V(\alpha_n)\circ \cdots \circ V(\alpha_1)$ for any path $\mu=\alpha_n \cdots \alpha_1$ of length $n \geq 1$.
Then, $V\in \rep Q$ is said to be a representation of $(Q,R)$ if
$V(\rho):=\sum\limits_{i=1}^{m} t_i V(\mu_i)=0$ for any $\rho=\sum\limits_{i=1}^{m} t_i \mu_i \in R$.
We denote by $\rep(Q,R)$ the full subcategory of $\rep Q$ consisting of the representations of $(Q,R)$.

The \emph{path category} $KQ$ of $Q$ over $K$ is defined as follows.
The objects of $KQ$ are the vertices of $Q_0$.
For each pair $(i,j)$ of objects of $KQ$, the morphisms from $i$ to $j$
are the linear combinations of paths from $i$ to $j$.
The composition of $KQ$ is defined as the bilinearization of the concatenation of paths.
Then for each object $i$ of $KQ$, the identity morphism of $i$ is given
as the path $e_i$ of length 0 at $i$.
Note that the obtained category $KQ$ naturally becomes a $K$-category,
in the sense that $KQ(i,j)$ are $K$-vector spaces for all $i, j \in Q_0$, and
the composition is $K$-bilinear.
\EDIT{For a bound quiver $(Q,R)$, we denote} the factor category $KQ/\ang{R}$ 
by $K(Q,R)$, where
$\ang{R}$ is the ideal of the $K$-category $KQ$ generated by $R$.
For instance, this notation is used later for $(Q, R) = \Gf{m,n}$
in Section 4 (see Definition \ref{dfn:ss--cc}).
For each morphism $\mu$ in $KQ$,
the morphism $\mu + \ang{R}$ in $KQ/\ang{R}$ is usually denoted just by $\mu$,
and for morphisms $\mu$ and $\nu$ in $KQ$, we regard $\mu = \nu$ in $KQ/\ang{R}$ if and only if $\mu - \nu \in \ang{R}$.

A $K$-linear functor from $K(Q,R)$ to $\vect_K$, the category
of finite-dimensional $K$-vector spaces, is called
a (left) $K(Q,R)$-\emph{module}, which can be identified with a representation of
$(Q,R)$ in an obvious way.
From this fact, representations of $(Q,R)$ are sometimes called
modules (over $K(Q,R)$).
A representation $M$ of $(Q,R)$ is said to be {\em indecomposable} if $M \cong M_1 \oplus M_2$ implies that $M_1 = 0$ or $M_2 =0$.

\begin{remark}
\label{rmk:poset-quiver}
In this work, we consider persistence modules as representations of bound quivers, except when comparing with the literature that uses poset representations. For the comparison with the literature, 
we here summarize the relationship between representations of a poset and those of a bound quiver.
Let $P$ be a locally finite poset (see Definition \ref{dfn:seg} for local finiteness).
\begin{enumerate}
\item
The {\em Hasse quiver} $H(P)$ of $P$ is a quiver defined as follows.
The set $H(P)_0$ of vertices is given by $H(P)_0:= P$,
and for any vertices $x, y$,
the set $H(P)(x,y)$ of arrows from $x$ to $y$ is given either
as a singleton $H(P)(x,y):=\{p_{y,x}\}$
if $x < y$ in $P$ and if there exist no $z \in P$ with $x < z < y$;
or $H(P)(x,y):=\emptyset$ otherwise.
We consider a bound quiver $(H(P), R(P))$, where $R(P)$ is the set of all commutativity relations in $H(P)$.
\item
When we regard $P$ as a category, we temporarily denote it by $C(P)$.
Denote by $K[C(P)]$ the $K$-linearization of $C(P)$.
Then $K[C(P)]$ is isomorphic to $K(H(P), R(P))$ as a $K$-linear category.
\item
A representation of $P$ (e.g., considered in \cite{kim2018generalized}) is defined to be a functor from $C(P)$ to $\vect_K$, which is uniquely extended to a $K$-linear functor from $K[C(P)]$ to $\vect_K$.
Hence by (2) above,
the category of representations of $P$ is isomorphic
to the category of left $K(H(P), R(P))$-modules, and hence to the category of representations of the bound quiver $(H(P), R(P))$.
In this way, the representations of a poset are covered by those of a bound quiver
(Keep this point of view in mind when reading Remarks \ref{remark:totkm} and \ref{remark:sskm}).

\end{enumerate}
\end{remark}

A fundamental result in representation theory is the  Krull-Schmidt theorem
(see \cite[Theorem 12.9]{anderson2012rings} or
\cite[I.4.10 Unique decomposition theorem]{assem2006elements}).

\begin{theorem}[Krull-Schmidt] \label{thm:KS}
Let $\calL$ be a complete set of representatives of isomorphism classes of indecomposable representations of a bound quiver $(Q,R)$.
For each representation $M$ of $(Q, R)$, 
there exists a unique function $d_M \colon \calL \to \bbZnn$ such that
$$M \cong \bigoplus_{L \in \calL} L^{d_M(L)}.$$
\end{theorem}
The function $d_M$ is called the \emph{multiplicity function} of $M$, and the value $d_M(L)$ the \emph{multiplicity} of the indecomposable $L$ in $M$.

As an example, let us consider the equioriented $A_n$-type quiver:
\[
  \Af{n}:
  \begin{tikzcd}
    1 \rar & 2 \rar & \cdots \rar & n
  \end{tikzcd}.
\]
It is known that in this case, $\calL$ is the set $\{\intv[b,d]\}_{1\leq b\leq d\leq n}$ of the so-called \emph{interval representations} $\intv[b,d]$ of $\Af{n}$ \cite{gabriel1972unzerlegbare}. The interval representation $\intv[b,d]$ is 
\[
  \intv[b, d] \colon 0 \longrightarrow \cdots
  \longrightarrow 0
  \longrightarrow \overset{b\text{-th}}{K} \longrightarrow K
  \longrightarrow \cdots
  \longrightarrow \overset{d\text{-th}}{K}
  \longrightarrow 0 \longrightarrow \cdots
  \longrightarrow 0,
\]
which has the vector space $\intv[b,d](i) = K$ at the vertices $i$ with $b\leq i\leq d$, and $0$ elsewhere, and where the maps between the neighboring vector spaces $K$ are identity maps and zero elsewhere. In the context of persistent homology \cite{edelsbrunner2002topological,edelsbrunner2008persistent}, a persistence module can be viewed as a representation of $\Af{n}$, and the multiplicity function $d_M$ encodes the information of the persistence diagram.

The underlying bound quiver we study in this work is the equioriented commutative grid $\Gf{m,n}$ defined below. Then, we consider $2$D persistence modules as representations of $\Gf{m,n}$. 
\begin{definition}[Equioriented commutative grid]
  Let $0 < m, n \in \mathbb{Z}$.  The bound quiver $\Gf{m,n}$, is defined to be the $2$D grid of size $m\times n$ with all horizontal arrows in the same direction and all vertical arrows in the same direction, together with full commutativity relations. It is also called the \emph{equioriented commutative grid} of size $m\times n$.
\end{definition}

For example, the equioriented $2\times 4$ commutative grid $\Gf{2,4}$ is the quiver
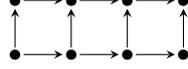
\[
  \begin{tikzcd}[graphstyle]
    \bullet \rar & \bullet \rar & \bullet \rar & \bullet \\
    \bullet \rar\uar & \bullet \rar\uar & \bullet\uar\rar & \bullet\uar
  \end{tikzcd}
\]
with full commutativity relations.

As mentioned in the introduction, for large enough size, $\Gf{m,n}$ is of wild representation type. That is, $\calL$ can be very complicated.
Instead, we consider a restricted class of representations, the interval-decomposable representations.
Following the notation in \cite{asashiba2022interval}, we first recall the definition of interval subquivers and interval representations for general bound quivers.
\begin{definition}[{Interval subquiver}]
  \label{def:intervalsubq}
  \leavevmode
  \begin{enumerate}
  \item Let $Q$ be a quiver. A full subquiver $Q'$ of $Q$ is said to be \emph{convex} in $Q$ if and only if for all vertices $x$, $y \in Q'_0$ and all vertices $z \in Q_0$, the existence of paths $x$ to $z$ and $z$ to $y$ in $Q$ imply that $z \in Q'_0$.

  \item A quiver Q is said to be \emph{connected} if it is connected as an “undirected graph”,

  \item A full subquiver $Q'$ of $Q$ is said to be an \emph{interval subquiver} of $Q$ if $Q'$ is convex (in $Q$) and connected.
  \end{enumerate}
\end{definition}

  Since an interval subquiver $I$ of $\Gf{m,n}$ is a full subquiver,
  (with $\Gf{m,n}$ fixed)
  $I$ is completely determined by its set of vertices $I_0$.
  Thus, we identify $I$ with its set of vertices $I_0$ where convenient.

For any two full subquivers $Q^\prime, Q^{\prime\prime}$ of $Q$,
the intersection $Q^\prime \cap Q^{\prime\prime}$ (respectively, the union $Q^\prime \cup Q^{\prime\prime}$) of $Q^\prime$ and $Q^{\prime\prime}$ is defined as the full subquiver of $Q$ having the vertex set $Q^\prime_0 \cap Q^{\prime\prime}_0$ (respectively, $Q^\prime_0 \cup Q^{\prime\prime}_0$).

Suppose that $Q^\prime$ and $Q^{\prime\prime}$ are interval subquivers of $Q$ with
$Q^\prime_0 \cap Q^{\prime\prime}_0 \not = \emptyset$. 
Note that $Q^{\prime} \cap Q^{\prime\prime}$ may not be connected, in general, and so may not be an interval. 
However, the following statement can be checked. 

\begin{lemma} \label{lem:union}
  Let $Q^\prime$ and $Q^{\prime\prime}$ be interval subquivers of $Q$.
  Then, $Q^\prime \cap Q^{\prime\prime}$ is a disjoint union of interval subquivers of $Q$.
\end{lemma}
\begin{proof}
To see this, we write $Q^\prime \cap Q^{\prime\prime}$ as a disjoint union of its connected components $C_i$ for $i=1,\cdots, n$ and show that each connected component $C_i$ is actually an interval subquiver of $Q$. It suffices to check that $C_i$ is convex.

For that, let $x$, $y$ be vertices of $C_i$ and $z$ a vertex of $Q$ such that there exist paths $x$ to $z$ and $z$ to $y$ in $Q$.
We show that $z$ is a vertex of $C_i$.

For each path $z = z_0\to z_1\to \cdots\to z_\ell = y$ in $Q$,
since $Q^{\prime}$ and $Q^{\prime\prime}$ are convex and $x,y$ are both in $Q^{\prime}$ and $Q^{\prime\prime}$, each $z_k$ is a vertex of $Q^{\prime}$ and $Q^{\prime\prime}$. Thus, all $z_k$ are vertices in $Q^{\prime} \cap Q^{\prime\prime}$ and the path $z = z_0\to z_1\to \cdots\to z_\ell = y$ is in fact a path in $Q^{\prime} \cap Q^{\prime\prime}$. 
Since $z_\ell = y \in C_i$ and $C_i$ is a connected component, we must have $z_0 = z \in C_i$. Thus $C_i$ is convex.
\end{proof}

On the other hand, 
$Q^\prime \cup Q^{\prime\prime}$ is not an interval subquiver in general,
even if $Q^\prime$ and $Q^{\prime\prime}$ are interval subquivers of $Q$ with $Q^\prime_0 \cap Q^{\prime\prime}_0 \not = \emptyset$. While connectedness is guaranteed since $Q^\prime_0 \cap Q^{\prime\prime}_0 \not = \emptyset$, convexity may fail to hold.

\begin{definition}
  For $0 < m,n\in\mathbb{Z}$, define $\II{m,n}$ to be the set of all nonempty interval subquivers of $\Gf{m,n}$.
\end{definition}

It is known that the interval subquivers of $\Gf{m,n}$ take on a distinctive ``staircase'' shape. See \cite{asashiba2022interval}. Below  is an example of an interval subquiver of $\Gf{4,6}$.
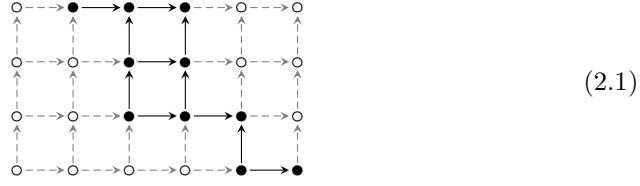
\begin{equation} \label{eq:st46}
  \newcommand{\bb}{\bullet}
  \newcommand{\rard}{\rar[dashed,gray]}
  \newcommand{\uard}{\uar[dashed,gray]}
  \begin{tikzcd}[graphstyle,every matrix/.append style={name=m}]
    \circ \rard & \bb \rar & \bb \rar & \bb \rard & \circ \rard & \circ \\
    \circ \rard\uard & \circ \rard\uard & \bb \rar\uar & \bb \rard\uar & \circ \rard\uard & \circ \uard\\
    \circ \rard\uard & \circ \rard\uard & \bb \rar\uar & \bb \rar\uar & \bb \rard\uard & \circ \uard\\
    \circ \rard\uard & \circ \rard\uard & \circ \rard\uard & \circ \rard\uard & \bb \rar\uar & \bb \uard
  \end{tikzcd}
\end{equation}
\EDIT{We also recall the example of a non-interval in \eqref{diag:noninterval}.}

Recall that for $M$ a representation of a bound quiver $(Q,R)$, the \emph{support} $\supp M$ of $M$ is the full subquiver of $Q$ with vertices $\{i\in Q \mid M(i)\not =0 \}$.
Finally, we are ready to recall the following generalization of interval representations of $\Af{n}$.



\begin{definition}[Interval representations]
\label{def:interval_rep}
Let $I$ be an interval subquiver of a quiver $Q$.
Then we define a representation $V_I$ of $Q$ as follows.
For each $x \in Q_0$ and each arrow $\alpha\colon x \to y$ in $Q$,
\[
V_I(x):=
\begin{cases}
K & \text{if $x \in I_0$},\\
0 & \text{otherwise},
\end{cases}
\quad \text{and}\quad
V_I(\alpha):=
\begin{cases}
1_K & \text{if $x, y \in I_0$},\\
0 & \text{otherwise}.
\end{cases}
\]
A representation of a bound quiver $(Q,R)$ is called an {\em interval representation} if it is isomorphic to $V_I$ for some interval subquiver $I$ of $Q$.
\end{definition}

Note that by construction, $V_I$ satisfies all the commutativity relations in $Q$.
It is obvious that $\supp V_I = I$.
For example, if $I$ is the interval subquiver of $\Gf{4,6}$ given by the quiver \eqref{eq:st46}, then the dimension vector of $V_I$ is given by  \eqref{eq:st46dv}.

A representation $M \in \rep(Q,R)$ is said to be \emph{interval-decomposable} if it can be expressed as a direct sum of interval representations. Equivalently, by Theorem~\ref{thm:KS}, $M$ is interval-decomposable if and only if  $d_M(L) = 0$ for all non-interval indecomposables $L$.

\subsection{Posets and Lattices}
In this subsection, we recall some basic definitions from poset and lattice theory. See \cite{stanley2011enumerative} for more details.

Recall that a poset (partially ordered set) $(P,\leq)$ is a set $P$ with partial order $\leq$. A poset $P$ is said to be {\em finite} if $P$ is finite as a set. 
The \emph{opposite} poset $P\op$ of $P$ is defined to be a poset $(P, \le\op)$, where for all $x, y \in P$, $x \le\op y$ if and only if $y \le x$.
Throughout this work, all posets are assumed to be finite.

\begin{definition}
\label{dfn:seg}
Let $P$ be a poset and $x,y\in P$.
The {\em segment} $[x,y]$ between $x$ and $y$ is defined to be 
$$
[x,y]:= \{z\in P \mid x \leq z \leq y   \}
$$
and define $\Seg(P)$ to be the set of all segments of $P$.
The poset $P$ is said to be {\em locally finite} if all segments of $P$ are finite sets.
The {\em open segment} $(x,y)$ between $x$ and $y$ is defined to be 
$$
(x,y):= \{z\in P \mid x < z < y   \}.
$$
It is clear that each segment of $P$ (respectively each open segment) of $P$ forms a subposet of $P$. 
We say that $y$ {\em covers} $x$ if $x < y$ and $(x,y) = \emptyset$.
The set of the elements covering $x$ is denoted by $\Cov (x)$. 
\end{definition}
We note that a segment $[x,y]$ is also called an interval in the literature,
but we do not use this term to avoid confusion.

\begin{definition}
Let $P$ be a poset and $S$ a subset of $P$. 
\begin{enumerate}
    \item An element $u\in P$ is said to be an {\em upper bound} of $S$ if $s\leq u$ for each $s\in S$. 
      The set of upper bounds of $S$ is denoted by $U(S)$.
      For a singleton $S = \{s\}$, we abuse the notation and write
      $U(s)$ for $U(\{s\})$.

    \item An element $x\in U(S)$ is said to be the {\em join} of $S$ if $x\leq u$ for each $u\in U(S)$. 
    Note that the join of $S$ is unique if it exists, and is denoted by $\bigvee S$.
    When $S=\{a,b\}$, then the join of $S$ is denoted by $a \vee b$.
\end{enumerate}
Dually, 
\begin{enumerate}
    \item[(3)] An element $l\in P$ is said to be an {\em lower bound} of $S$ if $l\leq s$ for each $s\in S$.
      The set of lower bounds of $S$ is denoted by $L(S)$.
      For a singleton $S = \{s\}$, we abuse the notation and write
      $L(s)$ for $L(\{s\})$.

    \item[(4)] An element $x\in L(S)$ is said to be the {\em meet} of $S$ if $l \leq x$ for each $l\in L(S)$. 
    Note that the meet of $S$ is unique if it exists, and is denoted by $\bigwedge S$.
    When $S=\{a,b\}$, then the meet of $S$ is denoted by $a \wedge b$.
\end{enumerate}
\end{definition}

\begin{definition}
Let $P$ be a poset.
\begin{enumerate} 
    \item $P$ is called a \emph{join-semilattice} (respectively, \emph{meet-semilattice}) if each two-element subset $\{a, b\}\subseteq P$ has a join (respectively, meet).
    \item $P$ is called a {\em lattice} if $P$ is a join-semilattice and a meet-semilattice.
    \item When $P$ is a lattice, $P$ is said to be {\em distributive} if 
    for all $x,y,z \in P$, 
    $$
    x \wedge (y \vee z) = (x \wedge y) \vee (x \wedge z)
    $$
    or equivalently, if for all $x,y,z \in P$,
    $$
    x \vee (y \wedge z) = (x \vee y) \wedge (x \vee z).
    $$
\end{enumerate}
\end{definition}

For a join-semilattice $P$ and $a, b, c \in P$, note that
$(a \vee b) \vee c =\bigvee\{a, b, c\} = a \vee (b \vee c)$. Thus the binary operation $\vee$  satisfies associativity, and hence generalized associativity.
Therefore in general, if $S=\{x_1,\hdots,x_n\}\subset P$, then 
\[
x_1 \vee x_2 \vee \cdots \vee x_n
\] 
is well-defined and equal to $\bigvee S$. A similar remark holds for $\bigwedge S$ in meet-semilattices.

The following fact is well-known and can be checked easily.
\begin{proposition} \label{prop:lattice}
If $P$ is a finite join-semilattice $($meet-semilattice$)$ with a lower bound $($upper bound\,$)$ of $P$, then $P$ is a lattice.
\end{proposition}

We will see later that the poset of intervals does not form a lattice globally, so we provide the following ``local'' definitions.
\begin{definition}
\ 
\begin{enumerate}
    \item A poset $P$ is called a {\em local lattice} if for any $x,y\in P$, 
    the segment $[x,y]$ is a lattice.
    \item A local lattice $P$ is said to be {\em locally distributive} if for any $x,y\in P$, 
    the segment $[x,y]$ is a distributive lattice. 
\end{enumerate}
\end{definition}

\subsection{The incidence algebra}
\label{subsec:incidence}
Let $F$ be a field
, and $P$ a
locally finite poset. Recall that $\Seg(P)$ is the set of segments of $P$. The \emph{incidence algebra} of $P$ over $F$ is the set of functions from $\Seg(P)$ to $F$, together with a ``pointwise'' $+$ operation, and convolution $*$ as the multiplication  operation. More precisely, for $f,g:\Seg(P) \rightarrow F$, define $f*g:\Seg(P) \rightarrow F$ by
\[
(f*g)([x,y]) := \sum\limits_{x\leq z \leq y} f([x,z])g([z,y]).
\]
Note that the sum above is finite because $P$ is locally finite, and hence $f*g$ is well-defined.
It can be shown that the incidence algebra of $P$ over $F$ is indeed an $F$-algebra, which we denote by $I(P)$. Its identity element is the delta function $\delta:\Seg(P)\rightarrow F$ with
 \[
   \delta([x,y]) = \left\{
     \begin{array}{ll}
       1_F & \text{if } x=y,\\
       0 & \text{otherwise.}
     \end{array}
   \right.
\]

\begin{remark}
\label{rmk:incidence}
For readers familiar with quiver representation theory, the following facts may be helpful to understand the incidence algebra.
\begin{enumerate}
\item
We can regard $I(P)$ as the $F$-algebra $A$ whose underlying vector space consists of
all infinite (if $P$ is an infinite poset) linear combinations
of symbols $[x, y] \in \Seg(P)$ by identifying each element $f \in I(P)$ with
\[
\sum_{[x,y] \in \Seg(P)}f([x,y])[x,y],
\]
having the multiplication defined first by setting
\[
[x,y][u,v]:= \begin{cases}
[x, v] & \text{if } y=u\\
0 & \text{if } y \ne u
\end{cases}
\]
for all $[x,y],[u,v] \in \Seg(P)$ and then extending to all of $I(P)$ bilinearly.
Note that the multiplication is well-defined by the local finiteness of $P$.
In particular, the identity element $\delta$
corresponds to the sum $\sum_{x \in P} [x,x]$.
\item
Therefore, in the case where $P$ is a finite poset,
$I(P) = A$ above is isomorphic to the matrix algebra\footnote{%
For a $F$-linear category $C$ with the set $C_0$ of objects finite, the matrix algebra $\Mat(C)$ of $C$ is defined to be the $F$-vector space $\bigoplus_{(y,x)\in C_0 \times C_0}C(x,y)$ of matrices with $(y,x)$-entries
in $C(x,y)$ together with the usual matrix multiplication. 
Then the category of (finite-dimensional) left $\Mat(C)$-modules is equivalent to the category of left $C$-modules.
}
of the category $F(H(P\op), R(P\op))$ (see Remark \ref{rmk:poset-quiver}),
where 
each $[x,y] \in \Seg(P)$ corresponds to
the coset $[x\leftarrow y]$ of a path from $y$ to $x$ in $H(P\op)$\ ($\cong H(P)\op$),
and the composite $[x,y][u,v] = \delta_{y,u}[x,v]$
corresponds to $[v\leftarrow u][y\leftarrow x] = \delta_{y,u}[v\leftarrow x]$
in $F(H(P), R(P))$.
Thus the category of (finite-dimensional) left $I(P)$-modules
is equivalent to
the category of left $F(H(P\op), R(P\op))$-modules, and hence
to the category of representations of the bound quiver $(H(P\op), R(P\op))$
over $F$.
%

%
\end{enumerate}
\end{remark}

\subsection{M\"obius Functions}
\label{subsec:mobius}

In this subsection, we assume that the characteristic of the field $F$ is zero, and
we review some basic facts about M\"obius functions. We refer the reader again to \cite{stanley2011enumerative} for more details.
In Sect.\ \ref{sec:approximation}, we apply the contents of this section in the setting that
$F = \mathbb{R}$ and $P = \II{m,n}$.

 \begin{definition}[Zeta and M\"obius functions]
 The \emph{zeta function} $\zeta:\Seg(P)\rightarrow F$ is the function with constant value $1_F$. Then, it can be shown that $\zeta$ is an invertible element of $I(P)$, with inverse called the \emph{M\"obius function} $\mu$.
\end{definition}

 Now, let $F^P$ be the set of all functions $P\rightarrow F$. Note that $F^P$ has a natural $F$-vector space structure by pointwise addition and scalar multiplication of functions.
 The incidence algebra $I(P)$ acts on $F^P$ from the left by the following.
 For each $f\in F^P$, $\phi\in I(P)$,
 define $\phi f \in F^P$ by
 \[
(\phi f)(x) := \sum\limits_{x\leq y}\phi([x,y])f(y) .
 \]
 It can be checked that $F^P$ is a left 
 $I(P)$-module
 with this left 
 action.
 For example, the computation
 \[
\begin{aligned}
(\psi(\phi f))(x) &= \sum\limits_{x\le y} \psi([x,y])(\phi f)(y) \\
&= \sum\limits_{x\le y}\psi([x,y])\left(\sum\limits_{y\le z}\phi([y,z]) f(z)\right)   \\
&= \sum\limits_{x \le y \le z} (\psi([x,y])\phi([y,z]))f(z) \\
&= \sum\limits_{x\le z}\left(\sum\limits_{x \le y \le z} \psi([x,y])\phi([y,z])\right) f(z)  \\
&= \sum\limits_{x\le z}(\psi*\phi)([x,z])  f(z)  \\
&= [(\psi * \phi)f](x),
\end{aligned}
\]
valid for all $f \in F^P,\ \phi, \psi \in I(P)$, and $x \in P$,
shows that this action is compatible with the multiplication (convolution) in $I(P)$.

\begin{remark}
Again for readers more familiar with quiver representation theory, we make the following comment.
Consider the case that $P$ is a finite poset and its Hasse quiver is connected.
By the equivalence of categories explained in Remark \ref{rmk:incidence}(2),
the left $I(P)$-module $F^P$ defined above corresponds to the interval representation
$V_{P\op}$ of the bound quiver $(H(P\op), R(P\op))$.
This point of view may be useful for understanding Equation~\eqref{eq:key_function} and the surrounding discussion.
\end{remark}


\section{Local lattice of intervals}
\label{sec:interval_lattice}

In this section, we study the set of isomorphism classes of interval representations for a fixed equioriented commutative $2$D grid $\Gf{m,n}$.
Note that an interval representation is uniquely defined (up to isomorphism) by its support, and thus it suffices to consider the set of interval subquivers $\II{m,n}$.
We also recall that
with $\Gf{m,n}$ fixed,
an interval subquiver $I$ is completely determined by its set of vertices $I_0$,
and we identify $I$ with its set of vertices $I_0$ where this does not cause any confusion.

First, we start with the following easy observation. 
\begin{proposition}
  With the order $\leq$ on $\II{m,n}$ defined by
  $I \leq I^\prime \Longleftrightarrow I \subseteq I^\prime$,
  $(\II{m,n},\leq)$ is a poset.
\end{proposition}
\begin{proof}
  This is immediate from the definitions.
\end{proof}

By Proposition~4.1 in \cite{asashiba2022interval}, each element $I$ of $\II{m,n}$ has a ``staircase'' form, which was denoted by:
\[
  I  = \bigsqcup_{i=s}^t [b_i,d_i]_i
\]
for some integers $1\leq s\leq t \leq m$ and some integers $1 \leq b_i \leq d_i \leq n$ for each $s\leq i\leq t$ such that
\begin{equation}
  \label{eq:intervalcond}
  b_{i+1}\leq b_{i}\leq d_{i+1}\leq d_{i} \text{ for all } i\in \{s,\dots,t-1\}.
\end{equation}
In this notation, each $[b_i,d_i]_i$ is the ``slice'' of the staircase at height $i$.
For example,
\EDIT{the interval $I = [5,6]_1 \sqcup [3,5]_2 \sqcup [3,4]_3 \sqcup [2,4]_4$
of $\Gf{4,6}$ can be visualized by the dimension vector of its corresponding interval representation:
\begin{equation} \label{eq:st46dv}
  \left(\smat{
      0&1&1&1&0&0 \\
      0&0&1&1&0&0 \\
      0&0&1&1&1&0 \\
      0&0&0&0&1&1
    }\right).
\end{equation}}
In general, the interval $I  = \bigsqcup_{i=s}^t [b_i,d_i]_i$ means that $I$ has vertices 
\[
  I_0 = \{ (i,x) \suchthat s\leq i \leq t, b_i \leq x \leq d_i \}.
\]
\EDIT{As above, we abuse the notation and
  use the corresponding dimension vector to denote the interval $I$.}

\begin{proposition}
  \label{prop:gradedposet}
  Let $I \in \II{m,n}$ and $J \in \Cov (I)$. Then, the number of vertices of $J$ is one more than that of $I$.
\end{proposition}

\begin{proof}[Sketch of Proof]
  Suppose that $I \subsetneq J$. We show that there exists a point $p \in J_0 \setminus I_0$ that can be added to $I$ to obtain an interval $I'$ with $I\subsetneq I' \subseteq J$.
  The result immediately follows from this, since if $J \in \Cov(I)$, then $J = I'$ by definition.
  That is, $J$ has one more vertex compared to $I$.

  Let
  \[
    I = \bigsqcup_{i=s}^t [b_i,d_i]_i \text{ and }
    J = \bigsqcup_{j=u}^v [c_j,e_j]_j.
  \]
  Since $I \subsetneq J$, it follows that $u\leq s \leq t \leq v$ and $c_k \leq b_k \leq d_k \leq e_k$ for each $k \in [s,t]$, in addition to the requirements for $I$ and $J$ to be intervals.
  We give below the point $p \in J_0 \setminus I_0$ that can be added to $I$ to obtain the interval $I'$.

  \begin{itemize}
  \item In case that $1 \leq u < s$,
    \begin{itemize}
    \item if $c_{s-1} \leq d_s$, then choose the point $p = (s-1,d_s)$;
    \item otherwise, if $c_{s-1} > d_s$, choose $p = (s,d_s+1)$.
    \end{itemize}
  \item The case $t < v \leq m$ is dual to the previous case.
    \begin{itemize}
    \item If $b_t \leq e_{t+1}$ choose $p = (t+1, b_t)$;
    \item otherwise, $p = (t,b_t-1)$ works.
    \end{itemize}
  \item Otherwise, we have $u = s \leq t = v$. In this case, we define
    \[
      L = \{k\in[s,t] \mid (k,b_k-1) \in J_0\}
      \text{ and }
      R = \{k\in[s,t] \mid (k,d_k+1) \in J_0\}.
    \]
    These are the row indices where a point to the left (and right, respectively) of $I$ is in $J$.
    Since $I \neq J$, it is clear that at least one of $L$ and $R$ is nonempty.
    \begin{itemize}
    \item If $L \neq \emptyset$, choose the point $p = (\max L, b_{\max L} -1)$.
    \item If $R \neq \emptyset$, choose the point $p = (\min R, d_{\min R} +1)$.
    \end{itemize}
  \end{itemize}
  For each of the cases above (which exhausts all possibilities), a routine check using the definitions shows that
  the chosen point $p$ can be added to $I$ to obtain an interval $I'$. This completes the proof.
\end{proof}

The above result implies that
$\II{m,n}$ is a graded poset with rank function
$\rho : \II{m,n} \rightarrow \mathbb{N}$ given by
$\rho(I) = \# I_0$, the number of vertices of $I$.

\begin{example}
  \label{example:cover2n}
  For any $n\in\bbN$ and any interval $I=[b_1,d_1]_1 \sqcup [b_2,d_2]_2  \in \II{2,n}$, 
  $\# \Cov (I) \leq 4$. 
  Indeed, any cover of $I$ takes on one of the following forms: 
  \[
    \begin{array}{l}
      [b_1 - 1,d_1]_1 \sqcup [b_2,d_2]_2,  \\ 
      {[} b_1, d_1 +1]_1 \sqcup [ b_2, d_2]_2,   \\ 
      {[} b_1, d_1]_1 \sqcup [ b_2 -1, d_2]_2, \text{ or}   \\ 
      {[} b_1, d_1]_1 \sqcup [ b_2,d_2 +1]_2.
    \end{array}
  \]
\end{example}

In general, we have the following, which follows immediately from Proposition~\ref{prop:gradedposet} and
the characterization of interval subquivers of $\Gf{m,n}$ as staircases.
\begin{proposition}
  \label{prop:cover}
  Let $I \in \II{m,n}$. Then,
  $\Cov(I) = \mathfrak{C} \cap \II{m,n}$
  where $\mathfrak{C}$ is the set of subquivers of $\Gf{m,n}$ obtained from $I$ by one of the following operations (if the result is a subquiver):
  \begin{enumerate}
  \item extending one row of $I$ by one adjacent vertex left of the row,
  \item extending one row of $I$ by one adjacent vertex right of the row,
  \item adding one vertex above the upper-left vertex of $I$, or
  \item adding one vertex below the lower-right vertex of $I$.
  \end{enumerate}
\end{proposition}
Let us express the above using the notation of
\[
  I  = \bigsqcup_{i=s}^t [b_i,d_i]_i
\]
for some integers $1\leq s\leq t \leq m$ and
some integers $1 \leq b_i \leq d_i \leq n$ for each $s\leq i\leq t$
such that $b_{i+1}\leq b_{i}\leq d_{i+1}\leq d_{i}$ for any $i\in \{s,\dots,t-1\}$.
Then $\Cov(I)$ is the set of \emph{valid interval subquivers} in the following
set of candidates $\mathfrak{C}$:
\begin{itemize}
\item  for $j \in \{s,\hdots, t\}$,
  \[
    \displaystyle\bigsqcup_{i=s}^t [b'_i,d_i]_i
    \text{, where }
    b'_i = \begin{cases}
      b_i-1 & \text{if } i=j,\\
      b_i & \text{otherwise,}
    \end{cases}
  \]
\item for $j \in \{s,\hdots, t\}$,
  \[
    \displaystyle\bigsqcup_{i=s}^t [b_i,d'_i]_i
    \text{, where }
    d'_i = \begin{cases}
      d_i+1 & \text{if } i=j,\\
      d_i & \text{otherwise,}
    \end{cases}
  \]
\item $\displaystyle\bigsqcup_{i=s}^t [b_i,d_i]_i \sqcup [b_t,b_t]_{t+1}$,
\item $\displaystyle [d_s,d_s]_{s-1} \sqcup \bigsqcup_{i=s}^t [b_i,d_i]_i$.
\end{itemize}
Note that some candidates may exceed the bounds of the commutative grid.
Those candidates are immediately disqualified.

\begin{example}
  \label{example:cover}
  We provide an example using the interval $I$  (filled-in circles):
  \[
    \newcommand{\bb}{\bullet}
    \newcommand{\rard}{\rar[dashed,gray]}
    \newcommand{\uard}{\uar[dashed,gray]}
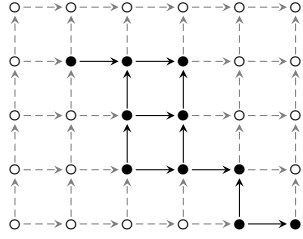
    \begin{tikzcd}[graphstyle,every matrix/.append style={name=m}]
      \circ \rard & \circ \rard & \circ \rard & \circ \rard & \circ \rard & \circ \\
      \circ \rard\uard & \bb \rar\uard & \bb \rar\uard & \bb \rard\uard & \circ \rard\uard & \circ\uard \\
      \circ \rard\uard & \circ \rard\uard & \bb \rar\uar & \bb \rard\uar & \circ \rard\uard & \circ \uard\\
      \circ \rard\uard & \circ \rard\uard & \bb \rar\uar & \bb \rar\uar & \bb \rard\uard & \circ \uard\\
      \circ \rard\uard & \circ \rard\uard & \circ \rard\uard & \circ \rard\uard & \bb \rar\uar & \bb \uard
    \end{tikzcd}
  \]
  in the commutative grid $\Gf{5,6}$.
  We illustrate the vertices in Proposition~\ref{prop:cover}.
  \begin{itemize}
  \item Vertices $v$ with $I_0 \cup \{v\} = C_0$ for some $C \in \Cov{I}$ are denoted with green check marks. These give all the cover elements $C$.
  \item The remaining vertices $v$ do not form cover elements. That is, there is no interval $C$ with $C_0 = I_0 \cup \{v\}$. These are denoted with red crosses. Note that two of them go out of bounds.
  \end{itemize}
  \[
    \newcommand{\bb}{\bullet}
    \newcommand{\cb}{{\color{green}\text{\cmark}}}
    \newcommand{\xb}{{\color{red}\text{\xmark}}}
    \newcommand{\rard}{\rar[dashed,gray]}
    \newcommand{\uard}{\uar[dashed,gray]}
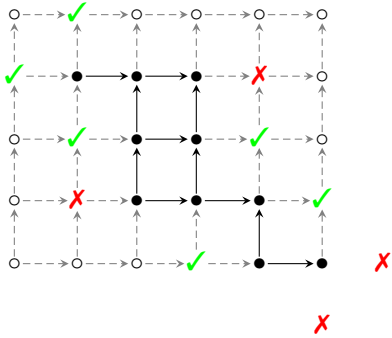
    \begin{tikzcd}[graphstyle,every matrix/.append style={name=m}]
      \circ \rard & \cb \rard & \circ \rard & \circ \rard & \circ \rard & \circ \\
      \cb \rard\uard & \bb \rar\uard & \bb \rar\uard & \bb \rard\uard & \xb \rard\uard & \circ\uard \\
      \circ \rard\uard & \cb \rard\uard & \bb \rar\uar & \bb \rard\uar & \cb \rard\uard & \circ \uard\\
      \circ \rard\uard & \xb \rard\uard & \bb \rar\uar & \bb \rar\uar & \bb \rard\uard & \cb \uard\\
      \circ \rard\uard & \circ \rard\uard & \circ \rard\uard & \cb \rard\uard & \bb \rar\uar & \bb \uard & \xb \\
      &&&&& \xb
    \end{tikzcd}
  \]
  Repeating the point above, each $C \in \Cov(I)$ is the unique interval subquiver $C$ with $I_0 \cup \{v\} = C_0$ for some vertex $v$ given by the green check marks.
\end{example}

\begin{proposition} \label{prop:locallattice}
  The poset $\II{m,n}$ is a local lattice.
\end{proposition}
\begin{proof}
Let $I,J$ be intervals of $\II{m,n}$ with $I \leq J$. We show that the segment $[I,J]$ is a lattice.

Let $J_1,J_2 \in [I,J]$.
Then, by Lemma~\ref{lem:union}, the intersection $J_1 \cap J_2$ is 
given by the disjoint union of some
intervals $C_i$: $\bigsqcup\limits_{i=1}^{l} C_i$.
In this setting, there exists a unique $j$ such that $C_j$ contains $I$.
Then the meet $J_1 \wedge J_2$ of $J_1$ and $J_2$ 
in the segment $[I,J]$
is exactly the interval $C_j$.
Proposition~\ref{prop:lattice} implies that the segment $[I,J]$ is a lattice.
\end{proof}

Note that in the above argument, the interval $J$ did not play any role in determining the meet
in $[I,J]$. We could have replaced $J$ by the maximum element $M$ in $\II{m,n}$,
which is the entire quiver of $\Gf{m,n}$.
That is, the meet of $J_1, J_2$ in $[I,J]$ is the same as the meet of $J_1, J_2$ in $[I,M] = U(I)$.
Thus, we also call the meet of $J_1, J_2$ in $[I,J]$ as the meet of $J_1, J_2$ over $I$.

On the other hand,  the join $J_1 \vee J_2$ in $[I,J]$ is
the minimum interval containing $J_1 \cup J_2$ by definition.
Clearly, $J_1 \cup J_2 \subset J \subset M$, and so
the join of $J_1, J_2$ in $[I,J]$ is the same as the join of $J_1, J_2$ in $[I,M] = U(I)$.
Thus, we also call the join of $J_1, J_2$ in $[I,J]$ as the join of $J_1, J_2$ over $I$.


\begin{example}
  Let
  $I = \left(\smat{0 & 1 & 0\\ 0 & 0 & 0}\right)$
  be an interval of $\II{2,3}$. The intervals
  $J=\left(\smat{0 & 1 & 1\\ 0 & 0 & 0}\right)$,
  $J^\prime=\left(\smat{0 & 1 & 0\\ 0 & 1 & 0}\right)$ in $U(I)$ have join
  $J \vee J^\prime = \left(\smat{0 & 1 & 1\\ 0 & 1 & 1}\right)$ over $I$.
\end{example}

While we have seen in Proposition~\ref{prop:locallattice} that $\II{m,n}$ is a local lattice, it is not a lattice as a whole (Example~\ref{ex:notlat}), nor is it locally distributive (Example~\ref{ex:notlocdist}).
\begin{example}
  \label{ex:notlat}
  In general, the meet and join is ill-defined.
  For example, let 
  $J = \left(\smat{1&0&0 \\ 0&0&0}\right)$ and 
  $J^\prime = \left(\smat{0&0&0\\0&0&1}\right)$ 
  be intervals in $\II{2,3}$.
  We note that $J \cap J^\prime = \emptyset$, so that there is no $I \in \II{m,n}$ with $J,J' \in U(I)$.  
  Then, 
  $X_1= \left(\smat{1 & 1 & 1\\ 0 & 0 & 1}\right)$
  and
  $X_2=\left(\smat{1 & 0 & 0\\ 1 & 1 & 1}\right)$
  are both minimal among intervals containing both $J$ and $J'$.
  Thus, $J \vee J^\prime$, which is supposed to be the minimum interval containing $J \cup J'$, is not well-defined. The poset $\II{m,n}$ is not a lattice, in general.
\end{example}

\begin{example}
  \label{ex:notlocdist}
  In general, the local lattice $\II{m,n}$ is not locally distributive.
  Indeed, let
  $I = \left(\smat{ 0 & 1 & 0 & 0\\ 0 & 0& 0 & 0}\right)$
  and
  $J = \left(\smat{1& 1 & 1 & 1\\1& 1 & 1 & 1}\right)$
  be intervals of $\II{2,4}$.
  Moreover, let
  $I_1= \left(\smat{1 & 1 & 0 & 0 \\ 1 & 1 & 1 & 0}\right)$,
  $I_2= \left(\smat{0 & 1 & 0 & 0\\ 0 & 1 & 1 & 1}\right)$,
  and
  $I_3= \left(\smat{0 & 1 & 1 & 1\\ 0 & 0 & 0 & 1}\right)$
  be intervals of the segment $[I,J]$. 
  Then we compute 
  $I_1 \vee (I_2 \wedge I_3) = I_1$ and
  $(I_1 \vee I_2) \wedge (I_1 \vee I_3) = \left(\smat{1 & 1 & 0 & 0\\ 1 & 1 & 1 & 1}\right) \neq I_1$.
\end{example}


\section{Compression and Compressed Multiplicities} \label{sec:comp}

In this section, we present the underlying mechanism for
\EDIT{``replacing'' (in Section~\ref{sec:approximation}) a persistence module by a related interval-decomposable object}.
Here, we define compression functors based on certain essential vertices. These compression functors then lead to what we call compressed multiplicities. We show that the well-known dimension vector and rank invariant are in fact special cases of compressed multiplicities. Furthermore, we show that for interval-decomposable representations, the true multiplicity information can be recovered from the compressed multiplicities.

\subsection{Essential Vertices}
First, we define two types of ``essential vertices''.

Recall that a vertex $x$ is said to be a \emph{source} if there are no arrows $\alpha$ with target $t(\alpha) = x$, and is said to be a \emph{sink} if there are no arrows $\alpha$ with source $s(\alpha) = x$.


\begin{definition}[Source-sink-essential vertices]
  Let $I$ be an interval subquiver of $\Gf{m,n}$. 
  A vertex $x \in I_0$ is said to be {\em source-sink-essential} (\emph{ss-essential}) if $x$ is a source or a sink in $I$.

  The set of ss-essential vertices of $I$ will be denoted by $I^{\rss}_0$.
\end{definition}

\begin{example} \label{ex:ssess}
  In the following interval subquiver $I$ in $\Gf{6,4}$:
  \[
    \newcommand{\rb}{{\circledast}}
    I=
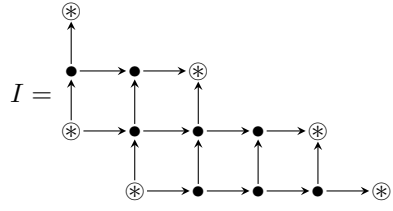
    \begin{tikzcd}[graphstyle]
      \rb & & \\
      \bullet \rar\uar &\bullet \rar & \rb \\
      \rb \rar\uar &\bullet \rar\uar & \bullet \rar\uar& \bullet \rar & \rb \\
      &\rb \rar\uar & \bullet \rar\uar & \bullet \uar \rar& \bullet \uar \rar & \rb,
    \end{tikzcd}
  \]
  the vertices denoted by $\circledast$ are ss-essential vertices of $I$. 
\end{example}

\begin{lemma} \label{lemcmp:ssess}
Let $I,J$ be intervals of $\II{m,n}$.  
Assume that $I^{\rss}_0 \subseteq J_0$.
Then we have $I \leq J$.
\end{lemma}
\begin{proof}
Let $x\in I_0$. 
Then, there is a source $y$, a sink $z$, and a path $\mu$ in $I$ from $y$ to $z$ such that $\mu$ passes through $x$.
Since $y,z\in I^{\rss}_0 \subseteq J_0$ and $J$ is convex, 
we have $x \in J_0$, as desired.
\end{proof}

\begin{definition}[Corner-complete-essential vertices]
Let $I$ be an interval subquiver of $\Gf{m,n}$. 
We set $I^{\mathrm{cc}}_0:= (\pr_1 I^{\rss}_0 \times \pr_2 I^{\rss}_0) \cap I_0$,
where $\pr_i : \bbZ \times \bbZ \to \bbZ$ is the projection map to the $i$-th 
\EDIT{coordinate}.
Elements of $I^{\mathrm{cc}}_0$ are said to be {\em corner-complete-essential} (\emph{cc-essential}),
and the full subquiver of $\Gf{m,n}$ given by this set is denoted by $I^{\mathrm{cc}}$.
\end{definition}

\begin{example} \label{ex:ccess}
For the interval subquiver $I$ used in Example~\ref{ex:ssess}:
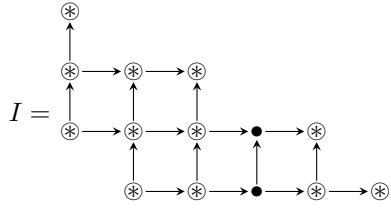
\[
  \newcommand{\rb}{{\circledast}}
  I=
  \begin{tikzcd}[graphstyle]
    \rb & & \\
    \rb \rar\uar &\rb \rar & \rb \\
    \rb \rar\uar &\rb \rar\uar & \rb \rar\uar &\bullet \rar & \rb \\
    &\rb \rar\uar & \rb \rar\uar & \bullet \uar \rar &\rb \uar \rar & \rb\\
  \end{tikzcd}
\]
the vertices denoted by $\circledast$ are cc-essential vertices of $I$. 
\end{example}

\begin{lemma} \label{lemcmp:ccess}
Let $I,J$ be intervals of $\II{m,n}$.
Assume that $I^{\rcc}_0 \subseteq J_0$.
Then we have $I \leq J$.
\end{lemma}
\begin{proof}
  Since $I^{\rss}_0\subseteq I^{\rcc}_0\subseteq J_0$, we have $I \leq J$ by Lemma~\ref{lemcmp:ssess}.
\end{proof}



\subsection{Compression}
In this subsection, we treat both types of essential vertices in parallel to define two types of compression of representations of the equioriented $2$D commutative grid $\Gf{m,n} = (Q, R)$.
In the previous subsection, we defined the sets of essential vertices $I^{\rss}_0$ and $I^{\rcc}_0$.
We consider the full subcategories of $K\Gf{m,n} = K(Q,R) = KQ/\ang{R}$ they induce.

\begin{definition}[ss-compressed category and cc-compressed category]
\label{dfn:ss--cc}
  Let $I$ be an interval subquiver of $\Gf{m,n}$ and $E$ be the set of all ss-essential vertices (or cc-essential vertices, respectively) of $I$.
The \emph{ss-compressed category} $\QComp{\rss}{I}$ (resp.~\emph{cc-compressed category} $\QComp{\rcc}{I}$) of $I$ is the full subcategory of $K \Gf{m,n}$ with set of objects $E$.
\end{definition}

For completeness, we also introduce the following concept, where we take all vertices of $I$ to be essential.
We use the designation ``$\tot$'' to stand for ``total'', since all vertices are considered essential in $I^\tot$.

\begin{definition}[compressed category]
The \emph{compressed category} $\QComp{\tot}{I}$ is the full subcategory of $K \Gf{m,n}$ consisting of all vertices of $I$. 
\end{definition}

\begin{remark}
For an interval $I$, we distinguish the following similar but different notions related to $I$: $I$ itself as a full subquiver of $\Gf{m,n}$, $V_I$ the representation of $K\Gf{m,n}$ with support $I$, and $\QComp{\tot}{I}$ as the full subcategory of $K\Gf{m,n}$ with objects the vertices of $I$.
\end{remark}

We note that the bound quiver of $\QComp{\tot}{I}$ is $(I,R_I)$ with the set of full commutativity relations $R_I$.
The ss-compressed category or cc-compressed category can also be expressed as a bound quiver, and we identify $\rep(Q_I^{\ast},R_I^{\ast}) \cong \rep \QComp{\ast}{I}$, where $(Q_I^{\ast},R_I^{\ast})$ is the bound quiver of the compressed category $\QComp{\ast}{I}$ for $\ast=\rss, \rcc,\tot$.

Throughout the rest of this work, we shall use the symbol `$\ast$' to stand for either `$\rss$', `$\rcc$' or `$\tot$' for statements that apply to all three cases as long as it does not cause any confusion.


\begin{example}
  \label{ex:compcats}
  For the interval subquiver $I$ in Example~\ref{ex:ssess}, the compressed categories (displayed as bound quivers) are the following:
  \[
    \newcommand{\rb}{{\circledast}}
    \QComp{\rss}{I} :
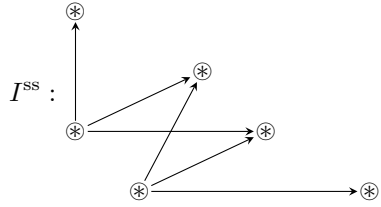
    \begin{tikzcd}[graphstyle]
      \rb                        &                             & \\
      &                             & \rb \\
      \rb\ar[rrr]\ar[uu]\ar[urr] &                             & & \rb \\
      &\rb\ar[rrrr]\ar[urr]\ar[uur] & & & &\rb
    \end{tikzcd}
  \]
  and
  \[
    \newcommand{\rb}{{\circledast}}
    \QComp{\rcc}{I} :  
    
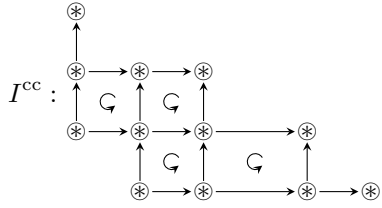
\begin{tikzcd}[graphstyle,every matrix/.append style={name=m},
      execute at end picture={
        \node (c1) at ($(m-3-1.center)!0.5!(m-2-2.center)$) {};
        \node (c2) at ($(m-3-2.center)!0.5!(m-2-3.center)$) {};
        \node (c3) at ($(m-4-2.center)!0.5!(m-3-3.center)$) {};
        \node (c4) at ($(m-4-3.center)!0.5!(m-3-5.center)$) {};
        \foreach \x in {c1,c2,c3,c4}{
          \draw[-{stealth[flex=0.75]}]([shift=(30:0.3em)]\x) arc (30:330:0.3em);
        }
      }]
      \rb         &             &                & &            & \\
      \rb\rar\uar & \rb\rar     & \rb            & &            & \\
      \rb\rar\uar & \rb\rar\uar & \rb\ar[rr]\uar & &\rb         & \\
      & \rb\rar\uar & \rb\ar[rr]\uar & &\rb\uar\rar & \rb 
    \end{tikzcd}
  \]
  while
  \[
    \newcommand{\rb}{{\circledast}}
   \QComp{\tot}{I} :  
    
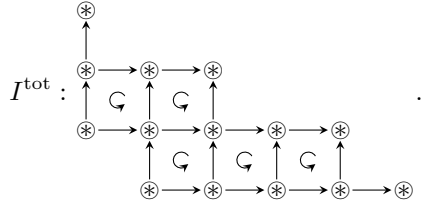
\begin{tikzcd}[graphstyle,every matrix/.append style={name=m},
      execute at end picture={
        \node (c1) at ($(m-3-1.center)!0.5!(m-2-2.center)$) {};
        \node (c2) at ($(m-3-2.center)!0.5!(m-2-3.center)$) {};
        \node (c3) at ($(m-4-2.center)!0.5!(m-3-3.center)$) {};
        \node (c4) at ($(m-4-3.center)!0.5!(m-3-4.center)$) {};
        \node (c5) at ($(m-4-4.center)!0.5!(m-3-5.center)$) {};
        \foreach \x in {c1,c2,c3,c4,c5}{
          \draw[-{stealth[flex=0.75]}]([shift=(30:0.3em)]\x) arc (30:330:0.3em);
        }
      }]
      \rb         &             &                & &            & \\
      \rb\rar\uar & \rb\rar     & \rb            & &            & \\
      \rb\rar\uar & \rb\rar\uar & \rb\ar[r]\uar & \rb\ar[r] &\rb         & \\
      & \rb\rar\uar & \rb\ar[r]\uar & \rb\ar[r]\uar &\rb\uar\rar & \rb 
    \end{tikzcd}.
  \]
\end{example}

\begin{definition}[Compression functor]
  \label{defn:compressionfunctor}
  Let $I$ be an interval subquiver of $\Gf{m,n}$ and let
  $\iota^{\rss}_I: \QComp{\rss}{I} \hookrightarrow K\Gf{m,n}$ (or $\iota^{\rcc}_I: \QComp{\rcc}{I} \hookrightarrow K\Gf{m,n}$, or $\iota_I^\tot: \QComp{\tot}{I} \hookrightarrow K\Gf{m,n}$, respectively) be the inclusion functor into the equioriented $2$D commutative grid.
  
  The \emph{ss-compression functor}
  $\VComp{\rss}{I}{\blank}:\rep K\Gf{m,n} \to \rep\QComp{\rss}{I}$
  (the \emph{cc-compression functor}
  $\VComp{\rcc}{I}{\blank}$ 
  or the \emph{tot-compression functor}
  $\VComp{\tot}{I}{\blank}$,
  respectively)
  is defined by $\VComp{\rss}{I}{M} = M \circ \iota^{\rss}_I$ ($\VComp{\rcc}{I}{M} = M \circ \iota^{\rcc}_I$ or $\VComp{\tot}{I}{M} = M \circ \iota_I^\tot$, respectively). That is,
 \[
 \VComp{\ast}{I}{M} = M \circ \iota^{\ast}_I
 \]
  for $\ast=\rss, \rcc,\tot$.
  
  Note that these functors are exactly the restriction functors.
\end{definition}

It is clear that the ss-compression, cc-compression, and tot-compression functors are additive by definition. To simplify the notation, we let $\VComp{\ast}{I}{\blank}$ stand for  $\VComp{\rss}{I}{\blank}$, $\VComp{\rcc}{I}{\blank}$, or $\VComp{\tot}{I}{\blank}$ for statements that hold for all three versions of compression.

Given $M \in \rep\Gf{m,n}$, the compressed representation $\VComp{\ast}{I}{M}$ is a representation of $\QComp{\ast}{I}$. Similary, the interval representation $V_I$ associated to the interval $I$ has a compressed representation $\VComp{\ast}{I}{V_I}$.
For example, the interval $I$ in Example~\ref{ex:ssess} is associated to the interval representation
\[
  \phantom{\VComp{ss}{I}{V_I}:}\mathllap{V_I:}
  \begin{tikzcd}[graphstyle]
    K \rar           &0 \rar             &0 \rar            &0 \rar             &0 \rar             &0 \\
    K \rar{1}\uar{1} &K \rar{1}\uar      & {K}\rar\uar      & 0\rar\uar         & 0 \rar \uar       & 0\uar \\
    K\rar{1}\uar{1}  &K \rar{1}\uar{1}   & K \rar{1}\uar{1} & K \rar{1}\uar     & {K} \rar\uar      & 0\uar \\
    0\rar\uar        &{K} \rar{1}\uar{1} & K \rar{1}\uar{1} & K \uar{1} \rar{1} & K \uar{1} \rar{1} & {K}\uar
  \end{tikzcd}
\]
which has ss-compressed representation (a representation of $\QComp{\rss}{I}$):
\[
  \VComp{ss}{I}{V_I}:
  \begin{tikzcd}[graphstyle]
    K                                    &                                                  &   &   &             & \\
                                         &                                                  & K &   &             & \\
    K\ar[rrr,"1"]\ar[uu,"1"]\ar[urr,"1"] &                                                  &   & K &             & \\
                                         &K\ar[rrrr,"1"]\ar[urr,"1"]\ar[uur,near start,"1"] &   &   & \phantom{K} & K \mathrlap{.}
  \end{tikzcd}
\]

While the compressed representation $\VComp{\ast}{I}{M}$ may be interesting in its own right, in the next definition we only consider the multiplicity of $\VComp{\ast}{I}{V_I}$ in $\VComp{\ast}{I}{M}$.

\begin{definition}[Compressed multiplicities]
  \label{defn:compressedmultiplicities}
Let $M$ be a representation of $\Gf{m,n}$ and $I \in \II{m,n}$.
Define the source-sink ($\rss$)-compressed multiplicity \EDIT{of $I$ in $M$} as
\[
\dbar{\rss}{M}{I}:= \nd{\VComp{\rss}{I}{M}}{\VComp{\rss}{I}{V_I}}.
\]
While not the main focus of this paper, for completeness we also define the \EDIT{corner-complete ($\rcc$) and total ($\tot$)} compressed multiplicities \EDIT{of $I$ in $M$}
\[
\dbar{\rcc}{M}{I}:= \nd{\VComp{\rcc}{I}{M}}{\VComp{\rcc}{I}{V_I}}, 
\]
and
\[
\dbar{\tot}{M}{I}:= \nd{\VComp{\tot}{I}{M}}{\VComp{\tot}{I}{V_I}}.
\]
In the above, $\nd{?}{\blank}$ is the usual multiplicity function obtained from Theorem~\ref{thm:KS}.
\end{definition}

One motivation for the above definitions is that
we want to compute the multiplicity of an interval module $V_I$ as a direct summand of $M$.
However, as this may not be straightforward,
we instead compute the multiplicity with respect to compressed versions of $M$ and $I$.
The rest of this section is devoted to exploring the consequences of this approach.

\begin{remark}
\label{remark:totkm}
Let
$\mathrm{rk}(M):\mathbf{Con}(P) \rightarrow \mathcal{J}(\mathcal{C})$
be the generalized rank invariant as defined in \cite{kim2018generalized}, applied to the setting we consider.  That is, $P$ is the poset corresponding to the $m\times n$ commutative grid, and the target set
is $\mathcal{J}(\mathcal{C}) = \mathcal{J}(\vect_K)$,
the set
of isomorphism classes of $K$-vector spaces.
By definition $\mathbf{Con}(P)$ is the set of path-connected subposets of $P$, which contains the set of intervals. See \cite{kim2018generalized} for more detailed definitions.
We note that for $I \in \II{m,n}$, the equality
\[
    \dbar{\tot}{M}{I} = \dim \mathrm{rk}(M)(I)
\]
holds.
This follows immediately from Lemma~3.1 of \cite{chambers2018persistent} applied to $\VComp{\tot}{I}{M}$. 
That is, for intervals $I$, the $\tot$-compressed multiplicity coincides with the generalized rank invariant of \cite{kim2018generalized}.
\end{remark}

As the next example shows, the values of
$\dbar{\rss}{M}{I}$ and 
$\dbar{\tot}{M}{I}= \dim \mathrm{rk}(M)(I)$
can be different in general.

\begin{example}
  \label{example:unequal}
Let $M$ be the representation of $\Gf{2,3}$ given by 
\[
  \begin{tikzcd}[ampersand replacement=\&]
    K \rar{\left[\smat{1\\1}\right]} \&
    K^2 \rar{\left[\smat{0 &1}\right]} \&
    K
    \\
    0 \rar \uar \&
    K \rar{1} \uar{\left[\smat{0\\1}\right]}  \&
    K \uar{1}
  \end{tikzcd}
\]
For the interval 
\[
I:
\begin{tikzcd}[graphstyle]
    \bullet \rar &\bullet \rar & \bullet \\
    &\bullet \rar\uar & \bullet \uar 
  \end{tikzcd}
\]
it can be computed that 
$\dbar{\rss}{M}{I} = 1$
while
$\dbar{\tot}{M}{I} = 0$.
\end{example}

\begin{remark}
\label{remark:sskm}
If we allow to change the form of the ``input'' to the function
$\dim \mathrm{rk}(M)(\text{-})$ and broaden its domain of definition, the equality
$
    \dbar{\rss}{M}{I} = \dim \mathrm{rk}(M)(\mathrm{Source}(I) \cup \mathrm{Sink}(I))
$
holds by the same reasoning as the previous remark.
Note that in general, $\mathrm{Source}(I) \cup \mathrm{Sink}(I)$ is not necessarily a path-connected subposet (\cite[Definition 2.16]{kim2018generalized}), and thus the original definition of the generalized rank invariant cannot be used.
That is, the values of the source-sink compressed multiplicity can be expressed as some value of the generalized rank invariant suitably generalized.
\end{remark}

\subsection{Rank invariant and dimension vector as compression}
In this subsection, we show that the compressed multiplicity generalizes the rank invariant \cite{carlsson2009theory}, a well-known invariant for $2$D persistence modules.

Recall that the \emph{rank invariant} is the function assigning to each pair $s,t\in \Gf{m,n}$ with a path from $s$ to $t$, the value
\[
  \rank(M(s\rightarrow t))
\]
where $M(s\rightarrow t):M(s) \rightarrow M(t)$ is the linear map associated by $M$ to a path from $s$ to $t$. Note that this is well-defined due to the commutativity relations imposed on $M$.

An interval $R =\bigsqcup_{i=x}^{y} [b_i,d_i]_i \in \II{m,n}$ is said to be a \emph{rectangle} if 
there exist $b$, $d$ ($1 \leq b \leq d \leq n$) such that $b_i = b$ and $d_i = d$ for any $i=x,\cdots,y$.
That is, 
$R =\bigsqcup_{i=x}^{y} [b,d]_i$.
The set of rectangles in $\II{m,n}$ is denoted by $R_{m,n}$. It is immediate that any rectangle $R$ has a unique source $s$  and a unique sink $t$.
Below is an example of a rectangle together with its source and sink.
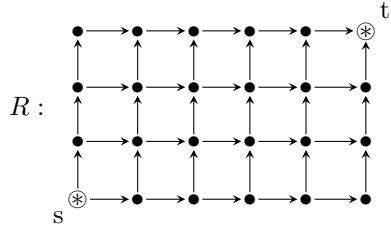
\[
  \newcommand{\bb}{\bullet}
  \newcommand{\rb}{{\circledast}}
  R:  
  \begin{tikzcd}[graphstyle,every matrix/.append style={name=m},
    execute at end picture={
      \node[anchor=north east] at (m-4-1.south west){{s}};
      \node[anchor=south west] at (m-1-6.north east){{t}};
      }
    ]
    \bb \rar & \bb \rar & \bb \rar & \bb \rar & \bb \rar & \rb \\
    \bb \rar\uar & \bb \rar\uar & \bb \rar\uar & \bb \rar\uar & \bb \rar\uar & \bb \uar\\
    \bb \rar\uar & \bb \rar\uar & \bb \rar\uar & \bb \rar\uar & \bb \rar\uar & \bb \uar\\
    \rb \rar\uar & \bb \rar\uar & \bb \rar\uar & \bb \rar\uar & \bb \rar\uar & \bb \uar
  \end{tikzcd}
\]
We comment that if $\Gf{m,n}$ is viewed as a subposet of $\mathbb{Z}\times \mathbb{Z}$ with coordinate-wise $\leq$,
the rectangle $R$ is in fact the segment $R = [s,t]$ in the poset $\mathbb{Z}\times \mathbb{Z}$.
In this work, we do not directly use this point of view since we defined $\Gf{m,n}$ as a bound quiver and not as a poset.

Conversely, given any pair $s,t\in \Gf{m,n}$ with a path from $s$ to $t$ (as in the definition of the rank invariant), there is a unique rectangle $R$ with source $s$ and sink $t$. Thus, the rank invariant can be equivalently defined as the function assigning to each rectangle $R$ in $\II{m,n}$ the value $\rank(M(s\rightarrow t))$, where $s$ is the unique source of $R$ and $t$ the unique sink.

Let $R$ be a rectangle with source $s$ and sink $t$. Let us compute the values of the compressed multiplicities at $R$.
\begin{itemize}
\item The ss-compressed category of $R$ is:
  $
  R^{\rss}:
  \begin{tikzcd}[graphstyle]
    s \rar & t
  \end{tikzcd},
  $
  so that $\VComp{\rss}{R}{M}$ is
  $
  \begin{tikzcd}[graphstyle,column sep=3em]
    M(s) \rar{M(s\rightarrow t)} & M(t)
  \end{tikzcd}.
  $
  Note that a linear map $f \colon V \to W$ between finite-dimensional vector spaces is equivalent to
  the direct sum $(K \to 0)^{\dim \ker f} \oplus (K \xrightarrow{1} K)^{\rank f} \oplus(0 \to K)^{\dim \mathrm{coker}\, f}$. Then
  we compute
  \[
    \begin{array}{rcl}
      \dbar{\rss}{M}{R}
      &=& \nd{{
          \left(
          \VComp{\rss}{R}{M}        
          \right)}}
          {{
          \VComp{\rss}{R}{V_R}
          }}\\    
      &=& \nd{{
          \left(
          \begin{tikzcd}[graphstyle,ampersand replacement=\&,column sep=3em]
            M(s) \rar{M(s\rightarrow t)} \& M(t)
          \end{tikzcd}
          \right)}}
          {{
          \begin{tikzcd}[graphstyle, ampersand replacement=\&]
            K \rar{1} \& K
          \end{tikzcd}
          }}\\
      &=& \rank(M(s\rightarrow t)).
    \end{array}
  \]

\item
\EDIT{
  Next, let us show that $\dbar{\tot}{M}{R} = \rank(M(s\rightarrow t))$.
  Note that via the equality with the generalized rank invariant (Remark~\ref{remark:totkm}),
  this is already known (see for example, \cite[Example~3.6(iii)]{kim2018generalized}), but for completeness
  we provide a proof.
For simplicity, put here $M' := \VComp{\tot}{R}{M}$.
Then 
$M'$ is the representation of $R^\tot = R$ obtained by restricting $M$ to the rectangle $R$.
Furthermore,
$\VComp{\tot}{R}{V_R}$
is
isomorphic to both the injective indecomposable representation $I(t)$ of $R$ and to the projective indecomposable representation $P(s)$ of $R$ corresponding to the vertex $s$.
By applying \cite[Theorem 3]{Asashiba2017} to $\VComp{\tot}{R}{V_R} \cong I(t)$, we have
\begin{equation}
  \label{eq:d_R(M)}
  \dbar{\tot}{M}{R} =
  d_{M'}(I(t)) = \dim\Hom_{R}(I(t), M') - \dim\Hom_{R}(I(t)/\soc I(t), M'),
\end{equation}
where $\soc I(t)$ is the socle of $I(t)$, which is the sum of all simple submodules of $I(t)$ by definition.

Here, the first term is given by
\[
  \dim\Hom_{R}(I(t), M') = \dim\Hom_{R}(P(s), M') = \dim M'(s) = \dim M(s).
\]
For the second term, consider the canonical short exact sequence
\[
  0 \to \soc I(t) \xrightarrow{\mu} I(t) \xrightarrow{\varepsilon} I(t)/\soc I(t) \to 0
\]
in the category of representations of $R$.
By applying the (contravariant left-exact) functor $\Hom_{R}(\blank, M')$ to this sequence,
we have
the first isomorphism in the following calculation:
\[
  \begin{aligned}
    \Hom_{R}(I(t)/\soc I(t), M') &\cong \ker \Hom_{R}(\mu, M')\\
    &= \{f \in \Hom_{R}(I(t), M') \mid f\mu = 0\}\\
    &= \{f \in \Hom_{R}(I(t), M') \mid f(\soc I(t)) = 0\}\\
    &\overset{(\mathrm{a})}{=} \{f \in \Hom_{R}(P(s), M') \mid f(p_{t,s}) = 0\}\\
    &\overset{(\mathrm{b})}{\cong} \{m \in M'(s) \mid M'(p_{t,s})(m) = 0\}\\
    &= \ker M'(p_{t,s}),
  \end{aligned}
\]
where $p_{t,s}$ is the morphism of $R$ given by the path $s \to t$, the equality (a) follows from $\soc I(t) = K p_{t,s}$,
and the isomorphism (b) follows from the canonical isomorphism $\Hom_{R}(P(s), M') \cong M'(s)$.
Then,
\[
  \begin{aligned}
    \dim \ker M'(p_{t,s}) & = \dim M'(s) - \dim\operatorname{Im} M'(p_{t,s})\\
    &= \dim M'(s) - \rank M'(p_{t,s})\\
    &= \dim M(s) - \rank M(s \to t).
  \end{aligned}
\]
Therefore, we have 
\[
  \begin{aligned}
    \dbar{\tot}{M}{R} = d_{M'}(I(t)) &= \dim M(s) - (\dim M(s) -  \rank M(s \to t))\\
    &= \rank M(s \to t)
  \end{aligned}
\]
as claimed.
}

\item
  \EDIT{
    Finally, we show that
    $\dbar{\rcc}{M}{R} =
    \rank(M(s\rightarrow t))$.
  }
  Since $R$ has source $s$ and sink $t$ together with its two other corners (say $u$ and $w$) as its cc-essential vertices, 
the cc-compressed category of $R$ is:
\[
  \QComp{\rcc}{R}:
  \begin{tikzcd}[graphstyle,every matrix/.append style={name=m},
      execute at end picture={
        \node (c1) at ($(m-1-1.center)!0.5!(m-2-2.center)$) {};
        \foreach \x in {c1}{
          \draw[-{stealth[flex=0.75]}]([shift=(30:0.3em)]\x) arc (30:330:0.3em);
        }
      }]
    u \rar & t \\
    s \rar \uar & w \uar
  \end{tikzcd}
\]
so that $\EDIT{M' := {} }\VComp{\rcc}{R}{M}$ is
$$
  \begin{tikzcd}[graphstyle]
    M(u) \rar & M(t) \\
    M(s) \rar \uar & M(w) \uar
  \end{tikzcd}.
$$
Furthermore, $\VComp{\rcc}{R}{V_R}$ is the injective
indecomposable
representation $I(t)$ associated to the vertex $t$:
$$
I(t)=
  \begin{tikzcd}[graphstyle]
    K \rar{1} & K \\
    K \rar{1} \uar{1}& K \uar{1}
  \end{tikzcd}.
  $$
  \EDIT{The proof proceeds as in the total compressed multiplicity case,
    but this time computing over $\QComp{\rcc}{R}$ instead of over $\QComp{\tot}{R} = R$.}
By \cite[Theorem 3 (see also Example 3)]{Asashiba2017}
$$
  \begin{array}{rcl}
    \dbar{\rcc}{M}{R}
    &=& \nd{{
        \left(
        \begin{tikzcd}[graphstyle,ampersand replacement=\&]
            M(u) \rar \& M(t) \\
            M(s) \rar \uar \& M(w) \uar
        \end{tikzcd}
        \right)}}
        {{
        \begin{tikzcd}[graphstyle, ampersand replacement=\&]
            K \rar{1} \& K \\
            K \rar{1} \uar{1} \& K \uar{1}
        \end{tikzcd}
        }}\\
    &=& \dim \Hom_{\QComp{\rcc}{R}} (I(t), \EDIT{M'}) - \dim \Hom_{\QComp{\rcc}{R}} (I(t)/\soc I(t), \EDIT{M'}) \\
    &=& \dim M(s) - (\dim M(s) - \rank(M(s\rightarrow t))) \\
    &=& \rank(M(s\rightarrow t)).
  \end{array}
$$

\end{itemize}

The above considerations prove the following.
\begin{proposition} \label{prop:rankinv}
Let $M$ be a representation of $\Gf{m,n}$ and $R$ a rectangle.
For $\ast=\rss,\rcc, \tot$, we have
$$
\dbar{\ast}{M}{R} = \rank M(s\to t),
$$
where $s$ is the unique source vertex of $R$ and $t$ is the unique sink vertex of $R$.
\end{proposition}

In this sense, the compressed multiplicities $\dbar{\ast}{M}{\blank}$ are generalizations of the rank invariant. With our invariant we hope to capture finer information that cannot be detected by just the rank invariant.

Next, we give an example of representations with the same rank invariants but different compressed multiplicities
for intervals that are not rectangles.
\begin{example}
Let $I = 
  \begin{tikzcd}[graphstyle]
    \bullet
    \rar 
    & \bullet  \\
    & \bullet 
    \uar[swap] 
  \end{tikzcd}
$
be an interval of $\Gf{2,2}= 
  \begin{tikzcd}[graphstyle]
    \bullet 
    \rar
    & \bullet  \\
    \bullet 
    \rar
    \uar
    & \bullet 
    \uar[swap]
  \end{tikzcd}
$.
Note that $I$ is not a rectangle.
We consider the following representations of $\Gf{2,2}$:
$$
M=
  \begin{tikzcd}[ampersand replacement=\&]
  K \rar{\left[\smat{1\\0}\right]} \& K^2 \\
  0 \rar \uar \& K \uar[swap]{\left[\smat{1\\0}\right]} 
  \end{tikzcd},\;\;
N =
  \begin{tikzcd}[ampersand replacement=\&]
  K \rar{\left[\smat{1\\0}\right]} \& K^2 \\
  0 \rar \uar \& K \uar[swap]{\left[\smat{0\\1}\right]}
  \end{tikzcd}.
$$
Clearly, rank invariants of $M$ and $N$ coincide.
However, we have $\dbar{\rss}{M}{I} = 1 \not = 0 = \dbar{\rss}{N}{I}$.
\end{example}


We end this subsection with the following observation.
\begin{proposition} \label{prop:dimvec}
Let $M$ be a representation of $\Gf{m,n}$ and $i$ a vertex of $\Gf{m,n}$.
For $\ast=\rss,\rcc, \tot$, we have 
\[
\dbar{\ast}{M}{\{i\}} = \dim M(i),
\]
where $\{i\}$ is the interval subquiver consisting of only the vertex $i$.
\end{proposition}
\begin{proof}
  A direct computation shows that
  \[
  \dbar{\ast}{M}{\{i\}}= \nd{\VComp{\ast}{\{i\}}{M}}{\VComp{\ast}{\{i\}}{V_{\{i\}}}}
  =\nd{M(i)}{K}=\dim M(i).
\]
Alternatively, this follows immediately from Proposition~\ref{prop:rankinv} by considering the rectangle with $s=t=i$.
\end{proof}
Namely, the compressed multiplicities $\dbar{\ast}{M}{-}$ restricted to vertices coincide with the dimension vector of $M$. 

\subsection{Compression and Inversion}

Next, we derive some basic properties of $\dbar{\ast}{M}{\blank}$, and end this section with Theorem~\ref{thm:interval}, which states that for \emph{interval-decomposable representations} $M$, we can recover the true multiplicity function $d_M$ using $\dbar{\ast}{M}{\blank}$.

First, we start with some Lemmas that lead to a Key Lemma~\ref{lem:key}.
\begin{lemma}\label{lem:dbar-decomp}
If a representation $M$ of $\Gf{m,n}$ decomposes as
$M = M_1 \oplus M_2$, then
$$\dbar{\ast}{M}{I} = \dbar{\ast}{M_1}{I} + \dbar{\ast}{M_2}{I}$$ 
for $\ast=\rss,\rcc,\tot$. 
\end{lemma}

\begin{proof}
Since the compression functor $\VComp{\ast}{I}{\blank}$ is additive, we have $\VComp{\ast}{I}{M} = \VComp{\ast}{I}{M_1} \oplus \VComp{\ast}{I}{M_2}$.
Then the statement follows by the Krull-Schmidt theorem.
\end{proof}

\begin{lemma}\label{lem:dbar-sen}
Let $I, J$ be intervals of $\Gf{m,n}$.
Then
\[
\dbar{\ast}{V_J}{I} = 
\begin{cases}
1 & \text{if } J \in U(I) \quad(\text{i.e.}~I \leq J),\\
0 & \text{otherwise}.
\end{cases}
\]
for $\ast=\rss,\rcc,\tot$. 
\end{lemma}
\begin{proof}
If $I \le J$, then $\VComp{\ast}{I}{V_J} = \VComp{\ast}{I}{V_I}$, thus
$\dbar{\ast}{V_J}{I} = 1$.

On the other hand, if $I \not\le J$, then there exists some $i \in I^{\ast}_0 \setminus J_0$ by Lemma~\ref{lemcmp:ssess} or Lemma~\ref{lemcmp:ccess} for $\ast=\rss,\rcc$, respectively, and by the fact that $I^\tot_0 = I_0$, for $\ast = \tot$. Thus, $i \in \supp(\VComp{\ast}{I}{V_I})$ but $i \not\in \supp(\VComp{\ast}{I}{V_J})$. This means that
 $\VComp{\ast}{I}{V_J}$ does not have a direct summand isomorphic to
$\VComp{\ast}{I}{V_I}$, showing that $\dbar{\ast}{V_J}{I} = 0$.
\end{proof}

\begin{lemma}[Key Lemma] \label{lem:key}
Let $M$ be an interval-decomposable representation of $\Gf{m,n}$ and $I$ an interval in $\II{m,n}$.
Then 
\[
\dbar{\ast}{M}{I} = \sum_{J \in U(I)} \nd{M}{V_J} 
\]
for $\ast=\rss,\rcc,\tot$.
\end{lemma}
\begin{proof}
Let $M \cong \bigoplus\limits_{J \in \II{m,n}} V_{J}^{d_M(V_J)}$ be an interval decomposition of a representation $M$ of $\Gf{m,n}$.
Then
\[
  \dbar{\ast}{M}{I} = \sum_{J \in \II{m,n}} d_M(V_J)\cdot\dbar{\ast}{V_{J}}{I} =
  \sum_{ J \in U(I)} \nd{M}{V_J}
\]
by Lemmas \ref{lem:dbar-decomp} and \ref{lem:dbar-sen}.
\end{proof}
As a consequence, in the case that $M$ is interval-decomposable, $\dbar{\ast}{M}{I}$ does not depend on $\ast$. 

Readers familiar with the M\"obius theory for (locally-finite) posets \cite{rota1964foundations} may recognize that Lemma~\ref{lem:key} simply states that for interval-decomposable representations, the function $\dbar{\ast}{M}{\blank}$ is equal to $d_M(\blank)$ multiplied by the zeta function. Theorem~\ref{thm:interval} below can then be seen as an application of M\"obius inversion.  Here, we give a direct proof of Theorem~\ref{thm:interval} and delay these M\"obius-theoretic considerations to a later section.

First, we note the following proposition which follows immediately from Lemma~\ref{lem:key}.
\begin{proposition}\label{prp:cmpd_M}
Let $M$ be an interval-decomposable representation of $\Gf{m,n}$ and $I$ an interval in $\II{m,n}$.
Then 
$$
\nd{M}{V_I} = \dbar{\ast}{M}{I} - \sum_{ J \in U(I) \backslash \{I\}} \nd{M}{V_J}.
$$
for $\ast=\rss,\rcc,\tot$.
\end{proposition}

\begin{theorem}[For interval-decomposables, compressed multiplicity recovers the multiplicity] \label{thm:interval}
Let $M$ be an interval decomposable representation of $\Gf{m,n}$ and $I$ an interval in $\II{m,n}$. Then:
\[  
    \nd{M}{V_I} = \dbar{\ast}{M}{I} + \sum\limits_{\smat{\emptyset \neq S \subseteq \Cov (I)}}(-1)^{\# S} \dbar{\ast}{M}{\bigvee S}.
\]
for $\ast=\rss,\rcc,\tot$.
\end{theorem}
\begin{proof}
  We define the function $f \colon 2^{U(I)} \to \bbZ$ by $f(S):= \sum\limits_{J \in S}d_M(V_J)$ for $S \in 2^{U(I)}$, where $2^{U(I)}$ is the power set of $U(I)$.
  Rewriting Proposition \ref{prp:cmpd_M}, we have
  \[
    d_M(V_I) = \dbar{\ast}{M}{I} - f\left(\bigcup_{J \in \Cov (I)}U(J)\right)
  \]
  since $U(I)\setminus\{I\} = \bigcup_{J \in \Cov (I)}U(J)$.
  Here, the inclusion-exclusion principle\footnote{More precisely, we use the inclusion-exclusion principle for finite measures, where we note that $(U(I), 2^{U(I)}, f)$ is a finite measure space.} shows that
  \[
    f\left(\bigcup_{J \in \Cov (I)}U(J)\right) =
    \sum_{\emptyset \neq S \subseteq \Cov (I)}(-1)^{(\# S - 1)}f\left(\bigcap_{J \in S} U(J)\right).
  \]
  By Proposition~\ref{prop:locallattice}, the join $\bigvee S$ in $U(I)$ exists, and it can be checked that
  \[
    \bigcap\limits_{J \in S} U(J) = U(\bigvee S)
  \]
  by definition.
  Therefore
  \[
    f\left(\bigcap\limits_{J \in S} U(J)\right) = f(U(\bigvee S)) =
    \dbar{\ast}{M}{\bigvee S}
  \]
  by Lemma \ref{lem:key}, which completes our proof.
\end{proof}

Theorem~\ref{thm:interval} says that to calculate $\nd{M}{V_I}$, it is enough to calculate $\dbar{\rss}{M}{J}$ (which is equal to $\dbar{\rcc}{M}{J}$ and also to $\dbar{\tot}{M}{J}$ since $M$ is interval-decomposable) for certain intervals $J$. We warn that the assumption that $M$ is interval-decomposable is necessary for Key Lemma~\ref{lem:key}, and so is also necessary here. It is easy to construct examples where the equality in Theorem~\ref{thm:interval} fails for non-interval-decomposable representations.

\begin{example}
  Let us follow the proof of Theorem~\ref{thm:interval} by computing a particular example.
  Let $M$ be an interval-decomposable representation of $\Gf{2,4}$ and let
  $I = \left(\smat{  0 & 1 & 1 & 0 \\ 0 & 1 & 1 & 0  }\right) \in \II{2,4}$, an interval.
  In this case, 
  \[
    \Cov (I) = \left\{
      I_1 := \left(\smat{  1 & 1 & 1 & 0 \\ 0 & 1 & 1 & 0  }\right),
      I_2 :=\left(\smat{  0 & 1 & 1 & 0 \\ 0 & 1 & 1 & 1  }\right)
    \right\}
  \]
  and 
  $
    I_1 \vee I_2 = \left(\smat{  1 & 1 & 1 & 0 \\ 0 & 1 & 1 & 1  }\right). 
  $
  By Lemma~\ref{lem:key}, we have
  \EDIT{
\[  
\begin{array}{rcl}
  \dbar{\ast}{M}{I} &=& \sum\limits_{J\in U(I)} d_M(V_J) \\
                    &=& \nd{M}{\left(\smat{  0 & 1 & 1 & 0 \\ 0 & 1 & 1 & 0  }\right)}
                       +\nd{M}{\left(\smat{  1 & 1 & 1 & 0 \\ 0 & 1 & 1 & 0  }\right)}
                       +\nd{M}{\left(\smat{  0 & 1 & 1 & 0 \\ 0 & 1 & 1 & 1  }\right)} \\
                    & & {}+\nd{M}{\left(\smat{  1 & 1 & 1 & 0 \\ 1 & 1 & 1 & 0  }\right)} 
                        +\nd{M}{\left(\smat{  1 & 1 & 1 & 0 \\ 0 & 1 & 1 & 1  }\right)} 
                        +\nd{M}{\left(\smat{  0 & 1 & 1 & 1 \\ 0 & 1 & 1 & 1  }\right)} \\
                    && {}+\nd{M}{\left(\smat{  1 & 1 & 1 & 0 \\ 1 & 1 & 1 & 1  }\right)}
                       +\nd{M}{\left(\smat{  0 & 1 & 1 & 1 \\ 0 & 1 & 1 & 1  }\right)} 
                       +\nd{M}{\left(\smat{  1 & 1 & 1 & 0 \\ 1 & 1 & 1 & 1  }\right)} \\
                    &=& \nd{M}{V_I} + \sum\limits_{J\in (U(I_1) \cup  U(I_2))} \nd{M}{V_J}) \\
                    &=& \nd{M}{V_I} + \sum\limits_{J \in U(I_1)}\nd{M}{V_J} +
                        \sum\limits_{J\in U(I_2)}\nd{M}{V_J} -
                        \sum\limits_{J\in (U(I_1) \cap U(I_2))} \nd{M}{V_J} \\
                    &=& \nd{M}{V_I} + \sum\limits_{J \in U(I_1)}\nd{M}{V_J} +
                        \sum\limits_{J\in U(I_2)}\nd{M}{V_J} -
                        \sum\limits_{J\in U(I_1 \vee I_2)} \nd{M}{V_J}.
\end{array}
\]
}
We thus have 
\[
  \nd{M}{V_I} = \dbar{\ast}{M}{I} - \dbar{\ast}{M}{I_1} - \dbar{\ast}{M}{I_2} + \dbar{\ast}{M}{I_1 \vee I_2}
\]
which is also given by Theorem~\ref{thm:interval}.
\end{example}

As another example, let us consider the equioriented $A_n$-type quiver, which can be viewed as $\Gf{1,n}$. In this setting, Theorem~\ref{thm:interval} reduces to the following well-known formula related to inclusion-exclusion. In fact, this perspective of using the inclusion-exclusion formula figured heavily in the early development of persistence diagrams, before the definition using indecomposables. See for example \cite{landi1997new}, \cite{frosini1999size}, \cite{robins1999towards}, \cite{cohen2005stability}, \cite{edelsbrunner2008persistent}, 
\cite{chazal2009proximity}, \cite{cerri2016hausdorff} and others.

\begin{corollary}
  Let $M \in \rep \Gf{1,n}$. 
  For $\intv[i,j]$ an interval representation of $\Gf{1,n}$,
  \[
    \begin{array}{rcl}
      \nd{M}{\intv[i,j]} &=& \left[\rank M((i-1)\rightarrow (j+1)) - \rank M((i-1)\rightarrow j)\right] - \\
                         && \left[\rank M(i\rightarrow (j+1)) - \rank M(i\rightarrow j)\right],
    \end{array}
  \]
  where if $i-1$ and/or $j+1$ is not in $\Gf{1,n}$, the corresponding term above is $0$.
\end{corollary}
\begin{proof}
  In $\Gf{1,n}$, it follows immediately from the definition that
  \[
    \dbar{\ast}{M}{\intv[i,j]} = \rank M(i \rightarrow j)
  \]
  for $\ast = \rss,\rcc,\tot$.  Furthermore, 
  $\Cov(\intv[i,j])$ contains  $\intv[i-1,j]$ if $i-1 \in \Gf{1,n}$ and contains $\intv[i,j+1]$ if $j+1 \in \Gf{1,n}$, and no other elements.

  It is well-known that all representations of $\Gf{1,n}$ are interval-decomposable, and thus Theorem~\ref{thm:interval} is applicable. Thus,
  \[
    \begin{array}{rcl}
      \nd{M}{\intv[i,j]} &=& \dbar{\ast}{M}{\intv[i,j]} \\
                         && {} - \dbar{\ast}{M}{\intv[i-1,j]} - \dbar{\ast}{M}{\intv[i,j+1]} \\
                         && {} + \dbar{\ast}{M}{\intv[i-1,j+1]},
    \end{array}
  \]
  where if $i-1$ and/or $j+1$ is not in $\Gf{1,n}$, the corresponding term above is $0$.
  Expanding and rearranging terms gives us the required expression.  
\end{proof}

We note that the same formula has been obtained by using Auslander-Reiten theory in the paper \cite{Asashiba2017} (Equation~(9) of \cite{Asashiba2017}). Our Theorem~\ref{thm:interval} here uses only the local lattice structure of $\II{m,n}$, and it may be interesting to explore   Theorem~\ref{thm:interval} using Auslander-Reiten theory, and more generally, a representation-theoretic perspective.


\subsection{Restriction to equioriented $2\times n$ commutative grid}
\label{subsec:2byn}

In this subsection, we study the special case of $\Gf{2,n}$, which is the equioriented commutative ladder.  In this setting, the compressed categories take on very nice forms.

\begin{proposition} \label{prop:ss-fin}
  Let $I \in \II{2,n}$. The quiver of the ss-compressed category $\QComp{ss}{I}$ has one of the following forms:
  \begin{enumerate}
  \item $\bullet$,
  \item $\begin{tikzcd}[graphstyle]
  \bullet \rar & \bullet
  \end{tikzcd}
  $,
  \item $\begin{tikzcd}[graphstyle]
  \bullet \rar & \bullet & \bullet \lar
  \end{tikzcd}
  $,
  \item $\begin{tikzcd}[graphstyle]
  \bullet  &\lar \bullet \rar & \bullet 
  \end{tikzcd}
  $,
  \item $\begin{tikzcd}[graphstyle]
  \bullet \rar & \bullet & \bullet \lar \rar & \bullet
  \end{tikzcd}
  $.
  \end{enumerate}
\end{proposition}
\begin{proof}
  A direct computation shows this.
\end{proof}

Similarly, we have the following.
\begin{proposition} \label{prop:cc-fin}
  Let $I \in \II{2,n}$.
  The bound quiver of the cc-compressed category $\QComp{cc}{I}$ has one of the following forms:
  \begin{multicols}{2}
  \begin{enumerate}
  \item $\bullet$,
  \item $\begin{tikzcd}[graphstyle]
  \bullet \rar & \bullet
  \end{tikzcd}
  $,
  \item $\begin{tikzcd}[graphstyle]
  \bullet \rar & \bullet & \bullet \lar
  \end{tikzcd}
  $,
  \item $\begin{tikzcd}[graphstyle]
  \bullet  &\lar \bullet \rar & \bullet 
  \end{tikzcd}
  $,
  \item $\begin{tikzcd}[graphstyle]
  \bullet \rar & \bullet & \bullet \lar \rar & \bullet
  \end{tikzcd}
  $.
  \item $
    \begin{tikzcd}[graphstyle,every matrix/.append style={name=m},
      execute at end picture={
        \node (c1) at ($(m-1-1.center)!0.5!(m-2-2.center)$) {};        
        \foreach \x in {c1}{
          \draw[-{stealth[flex=0.75]}]([shift=(30:0.3em)]\x) arc (30:330:0.3em);
        }
      }]
      \bullet \rar & \bullet \\
      \bullet \rar\uar & \bullet \uar
    \end{tikzcd}
    $, 
  \item $
    \begin{tikzcd}[graphstyle,every matrix/.append style={name=m},
      execute at end picture={
        \node (c1) at ($(m-1-2.center)!0.5!(m-2-3.center)$) {};        
        \foreach \x in {c1}{
          \draw[-{stealth[flex=0.75]}]([shift=(30:0.3em)]\x) arc (30:330:0.3em);
        }
      }]
      \bullet \rar &\bullet \rar & \bullet \\
      &\bullet \rar\uar & \bullet \uar 
    \end{tikzcd}$,
  \item $
    \begin{tikzcd}[graphstyle,every matrix/.append style={name=m},
      execute at end picture={
        \node (c1) at ($(m-1-1.center)!0.5!(m-2-2.center)$) {};        
        \foreach \x in {c1}{
          \draw[-{stealth[flex=0.75]}]([shift=(30:0.3em)]\x) arc (30:330:0.3em);
        }
      }]
      \bullet \rar & \bullet \\
      \bullet \rar\uar & \bullet \rar \uar & \bullet 
    \end{tikzcd}
    $,
  \item \label{item:biggestcc} $
    \begin{tikzcd}[graphstyle,every matrix/.append style={name=m},
      execute at end picture={
        \node (c1) at ($(m-1-2.center)!0.5!(m-2-3.center)$) {};        
        \foreach \x in {c1}{
          \draw[-{stealth[flex=0.75]}]([shift=(30:0.3em)]\x) arc (30:330:0.3em);
        }
      }]
      \bullet \rar&\bullet \rar & \bullet \\
      &\bullet \rar\uar & \bullet \rar\uar & \bullet 
    \end{tikzcd}
    $.
  \end{enumerate}
\end{multicols}
\end{proposition}
\begin{proof}
  It is immediate that there are at most $6$ cc-essential vertices, arranged in the shape of \ref{item:biggestcc}, for an interval in $\II{2,n}$. The rest of the forms cover the cases where some of those vertices are not cc-essential in $I$.
\end{proof}

For $I\in \II{2,n}$ with $n \geq 5$, $\QComp{\tot}{I}$ is of infinite representation type
(see \cite[Theorem~1.3]{bauer2020cotorsion} or \cite{escolar2016persistence} for example).
Therefore, it may be difficult to calculate the values $\dbar{\tot}{M}{I}$.

On the other hand, Proposition~\ref{prop:ss-fin} and Proposition~\ref{prop:cc-fin} show that  
$\QComp{\rss}{I}$ and $\QComp{\rcc}{I}$ are of finite type for any $I\in\II{2,n}$.
In addition, the Auslander-Reiten quivers for the bound quivers in the lists of Proposition~\ref{prop:ss-fin} and Proposition~\ref{prop:cc-fin} can be calculated explicitly.
Thus, it is {\em not} difficult to calculate the values $\dbar{\ast}{M}{I}$ for $\ast = \rss,\rcc$, in the setting of the equioriented $2\times n$ commutative grid.

We discuss more about computations in Section~\ref{sec:algorithm}.

\section{\EDIT{Interval-decomposable replacement}}
\label{sec:approximation}

In this section, let us discuss how to use the above ideas 
\EDIT{for replacing a}
general $2$D persistence modules in $\rep\Gf{m,n}$ by an interval-decomposable one. First, let us rephrase Theorem~\ref{thm:interval} using the language of M\"obius inversion, as discussed in Subsection~\ref{subsec:mobius}, with underlying field $F=\mathbb{R}$.

We can view $d_M$ and $\dbarfun{\ast}{M}$ as functions $\II{m,n}\rightarrow \mathbb{R}$ (taking only nonnegative integer values). For $d_M$, this is an abuse of notation, since $d_M$ is a function from (isomorphism classes of) \emph{all} indecomposables, but here we are using the symbol to denote it restricted to the interval representations of $\Gf{m,n}$, identified with the set of intervals $\II{m,n}$.

In the notation of Subsection~\ref{subsec:mobius}, we have $d_M, \dbarfun{\ast}{M} \in \mathbb{R}^{\II{m,n}}$. Then, the Key Lemma~\ref{lem:key} states that for $M$ interval-decomposable,
\begin{equation}
  \label{eq:key_function}
  \dbarfun{\ast}{M} = \zeta d_M
\end{equation}
where the multiplication of $\zeta$ in Eq.~\eqref{eq:key_function} is precisely the left action of $I(\II{m,n})$ on $\mathbb{R}^{\II{m,n}}$. 
By M\"obius inversion (multiplication of $\mu = \zeta^{-1}$), we obtain
\begin{equation}
  \label{eq:mobiusrephrase}
  d_M = \mu \dbarfun{\ast}{M}.
\end{equation}

This expresses $d_M$ in terms of $\dbarfun{\ast}{M}$, a conclusion similar to the one of  Theorem~\ref{thm:interval}. 
Next, we show that the coefficients appearing in Theorem~\ref{thm:interval} gives the values of the M\"obius function $\mu([I,J])$ of $\II{m,n}$. 
 \begin{definition}
   \label{defn:muprime}
   Define the function $\mu':\Seg(\II{m,n})\rightarrow \mathbb{R}$, an element of the incidence algebra $I(\II{m,n})$ by the following.
   \begin{equation}
     \label{eq:mobius}
     \mu'([I,J]) =
       \begin{cases}
         1 & \text{if } I = J,\\
         \displaystyle\sum\limits_{\substack{J={\bigvee S}\\ \emptyset \neq S \subseteq \Cov(I)}} (-1)^{\# S} & \text{otherwise.}
       \end{cases}
   \end{equation} 
 \end{definition}
 Note that in the case $I\neq J$ and where there is no $\emptyset \neq S \subseteq \Cov(I)$ such that $J=\bigvee S$, the sum above is empty, and thus $\mu'([I,J])= 0$. The values of $\mu'$ are exactly the coefficients appearing in the formula of Theorem~\ref{thm:interval}, from which we immediately get the following Corollary. 
 \begin{corollary}[Restatement of Theorem~\ref{thm:interval}] \label{cor:restate}
   Let $M$ be an interval-decomposable representation of $\Gf{m,n}$ and $I$ an interval in $\II{m,n}$. Then:
   \[  
     d_M = \mu' \dbarfun{\ast}{M}
   \]
   for $\ast=\rss,\rcc,\tot$.
 \end{corollary}

 \begin{theorem}
   \label{thm:mobiusequal}
   Let $\mu'$ be as defined in Definition~\ref{defn:muprime}, and $\mu$ be the M\"obius function of the poset $\II{m,n}$. Then,
   \[
     \mu = \mu'.
   \]
   In particular, Equation~\eqref{eq:mobius} gives the values of $\mu$.
 \end{theorem}
 \begin{proof}
   Let $I \leq L$ be intervals in $\II{m,n}$.
   Below, we compare the values $\mu([I,L])$ and $\mu'([I,L])$ by induction on $L$.

First, let us consider $L$ a cover of $I$ and fix $M=V_L$.
By Corollary~\ref{cor:restate} and Equation~\ref{eq:mobiusrephrase}, we have
\[
  \mu' \dbarfun{\ast}{M} = \mu \dbarfun{\ast}{M}.
\]
We obtain the following sequence of equations by working on both sides the equation.
\[
\begin{array}{rcl}
  (\mu' \dbarfun{\ast}{M})(I) &=& (\mu \dbarfun{\ast}{M})(I)
  \\ 
  \sum\limits_{I\leq J} \mu'([I,J]) \dbarfun{\ast}{M}(J)  &=& \sum\limits_{I\leq J} \mu([I,J]) \dbarfun{\ast}{M}(J)
  \\
  \sum\limits_{I\leq J\leq L} \mu'([I,J]) \dbarfun{\ast}{M}(J)  &=& \sum\limits_{I\leq J\leq L} \mu([I,J]) \dbarfun{\ast}{M}(J)
  \\
  \mu'([I,I]) + \mu'([I,L])  &=& \mu([I,I]) + \mu([I,L])
  \\
  1 + \mu'([I,L])  &=& 1 + \mu([I,L]),
\end{array}
\]
where going from the second line to the third line follows by Lemma~\ref{lem:dbar-sen}.
We conclude $\mu'([I,L])  =  \mu([I,L])$ for any $L \in \Cov (I)$. 

Next, 
we assume that for any interval $L^{\prime}$ with $L^{\prime} < L$,
$\mu'([I,L^{\prime}]) = \mu([I,L^{\prime}])$.
Then we have the following sequence of equations by taking $M=V_L$ and again using Lemma~\ref{lem:dbar-sen}:
\[
\begin{array}{rcl}
  (\mu' \dbarfun{\ast}{M})(I) &=& (\mu \dbarfun{\ast}{M})(I)
  \\ 
  \sum\limits_{I\leq J\leq L} \mu'([I,J])   &=& \sum\limits_{I\leq J\leq L} \mu([I,J])
  \\
  \sum\limits_{I\leq J < L} \mu'([I,J])  + \mu'([I,L])  &=& \sum\limits_{I\leq J < L} \mu([I,J]) + \mu([I,L]).
\end{array}
\]
Since we have $\sum\limits_{I\leq J < L} \mu'([I,J]) =\sum\limits_{I\leq J < L} \mu([I,J]) $ by the inductive assumption, 
we obtain $\mu'([I,L])=\mu([I,L])$.
By the induction, we get the conclusion.
\end{proof}

As we have seen, $d_M = \mu \dbarfun{\ast}{M}$ for $M$ interval-decomposable.
Even in the case where $M$ is not interval-decomposable, we nevertheless can do the transformation. Thus we \emph{define}
$ \aprx{\ast}{M} := \mu \dbarfun{\ast}{M}$ in general.

\begin{definition}
  \label{defn:dtilde}
  Put $\ast=\rss,\rcc,\tot$. Define $\aprx{\ast}{M} := \mu \dbarfun{\ast}{M}$. In particular, for each $I\in \II{m,n}$ an interval subquiver of $\Gf{m,n}$,
  \[
    \aprx{\ast}{M}(I):= \dbar{\ast}{M}{I} + \sum\limits_{\emptyset \not= S \subseteq \Cov (I)} (-1)^{\# S} \dbar{\ast}{M}{\bigvee S}.
  \]
\end{definition}

First, we note the following obvious property of $\aprx{\ast}{M} (\blank)$.
\begin{lemma}
  \label{lem:dtilde-decomp}
  If $M \cong M_1 \oplus M_2$, then we have 
  \[
  \aprx{\ast}{M} (\blank) = \aprx{\ast}{M_1} (\blank) + 
  \aprx{\ast}{M_2} (\blank).
  \]
\end{lemma}
\begin{proof}
  Since $\dbar{\ast}{M}{\blank} = \dbar{\ast}{M_1}{\blank}+\dbar{\ast}{M_2}{\blank}$ by Lemma~\ref{lem:dbar-decomp}, we have the desired equation by definition.
\end{proof}

Since in general
\[
  M \cong \bigoplus\limits_{X \in \calL} X^{d_M(X)}
\]
by Theorem~\ref{thm:KS} (where $d_M$ is the actual multiplicity function, not restricted to intervals),
one way of constructing an interval-decomposable object is to naively define
\begin{equation}
  \label{eq:naive}
  \aprxM{\ast}{M} = \bigoplus\limits_{I\in\II{m,n}} \left(V_I\right)^{\aprx{\ast}{M}(I)}
\end{equation}
by taking the function $\aprx{\ast}{M}$ on $\II{m,n}$ as a substitute for the function $d_M$ on $\calL$. Defined this way, $M \cong \aprxM{\ast}{M}$ for interval-decomposable $M$. However, the value $\aprx{\ast}{M}(I)$ can be negative in general, and thus the direct sum in Eq.~\eqref{eq:naive} does not make sense.

For example, we have the following.
\begin{example}
\label{ex:negativedtilde}
Let $M$ be the representation of $\Gf{2,3}$ given by 
\[
  \begin{tikzcd}[ampersand replacement=\&]
    K \rar{\left[\smat{1\\1}\right]} \&
    K^2 \rar{\left[\smat{0 &1}\right]} \&
    K
    \\
    0 \rar \uar \&
    K \rar{1} \uar{\left[\smat{0\\1}\right]}  \&
    K \uar{1}
  \end{tikzcd}
\]
The value of $\aprx{\rss}{M}(I)$ is $0$ except in the cases of $I$ being one of the intervals $I_1,I_2,I_3,I_4$ given below.
\begin{enumerate}
\item For $I_1:
  \begin{tikzcd}[graphstyle]
    \bullet \rar &\bullet &  \\
    &\bullet \rar\uar & \bullet 
  \end{tikzcd}$, $\aprx{\rss}{M} (I_1) =-1$,
\item
  For $I_2:
  \begin{tikzcd}[graphstyle]
    \bullet \rar &\bullet \rar & \bullet \\
    &\bullet \rar\uar & \bullet \uar 
  \end{tikzcd}$, $\aprx{\rss}{M}(I_2) = 1$,
\item 
  For $I_3:
  \begin{tikzcd}[graphstyle]
    \phantom{\bullet}&\bullet &  \\
    &\bullet \rar\uar & \bullet 
  \end{tikzcd}$, $\aprx{\rss}{M}(I_3) =1$,
\item 
  For $I_4:
  \begin{tikzcd}[graphstyle]
    \bullet \rar &\bullet &\phantom{\bullet}   \\
    &\phantom{\bullet} & \phantom{\bullet} 
  \end{tikzcd}$
  , $\aprx{\rss}{M}(I_4) =1$.
\end{enumerate}
\end{example}
\begin{proof}
  We directly use Definition~\ref{defn:dtilde} to compute $\aprx{\rss}{M} (I_1)$. We let $\Cov(I_1) = \{I_2,I_5\}$, and let $I_6 = I_2 \vee I_5$, where the intervals are given below. We first compute the value of the compressed multiplicity $\dbar{\rss}{M}{\blank}$ of these intervals.
  We have:
  \[
    I_1:
    \begin{tikzcd}[graphstyle]
      \bullet \rar &\bullet &  \\
      &\bullet \rar\uar & \bullet 
    \end{tikzcd},
    \ \dbar{\rss}{M}{I_1}= 0,
  \]  
  \[
    I_2:
    \begin{tikzcd}[graphstyle]
      \bullet \rar &\bullet \rar & \bullet \\
      &\bullet \rar\uar & \bullet \uar 
    \end{tikzcd},
    \ \dbar{\rss}{M}{I_2} = 1,
  \]
  \[
    I_5:
    \begin{tikzcd}[graphstyle]
      \bullet \rar &\bullet &  \\
      \bullet \rar \uar &\bullet \rar\uar & \bullet  
    \end{tikzcd},    
    \ \dbar{\rss}{M}{I_5} = 0,
  \]
    \[
    I_6:
    \begin{tikzcd}[graphstyle]
      \bullet \rar &\bullet \rar &  \bullet\\
      \bullet \rar \uar &\bullet \rar\uar & \bullet\uar
    \end{tikzcd},    
    \ \dbar{\rss}{M}{I_6} = 0.
  \]
  Thus, by definition,
  \[
    \aprx{\rss}{M}(I_1) = 0 -1 -0 + 0 = -1.
  \]

  The other computations follow similarly.
\end{proof}

For $M$ interval-decomposable, it is clear from the above that all values of $\aprx{\ast}{M}$ are nonnegative, as it is equal to $d_M$ itself. In the next example we see that the converse does not hold, and so we cannot use the nonnegativity of $\aprx{\ast}{M}$ to check for interval-decomposability.
\begin{example}[Continuation of Example~\ref{ex:negativedtilde}]
  There exists a persistence module $N$ over $\Gf{m,n}$ (for some $m, n$) such that $\aprx{\ast}{N}$ is nonnegative, but $N$ is not interval-decomposable.

  In particular, let $M$ and $I_i$ $(i=1,2,3,4)$  be as given in Example~\ref{ex:negativedtilde}. Then $N := M \oplus I_1$ is such an example.
\end{example}
\begin{proof}    
Since $N=M \oplus I_1 $, $\aprx{\rss}{N} = \aprx{\rss}{M} + \aprx{\rss}{I_1}$ by Lemma~\ref{lem:dtilde-decomp}.
Then we have
\[
  \aprx{\rss}{N}(I_1) = -1 + 1 =0
\]
and  $\aprx{\rss}{N}(I) = \aprx{\rss}{M}(I) + 0 \geq 0$ for all intervals $I \neq I_1$. Thus, $\aprx{\rss}{N}$ is nonnegative, but $N$ is not interval-decomposable since $M$ is an indecomposable summand of $N$ that is not isomorphic to an interval representation.
\end{proof}

To deal with the possibility of negative terms in $\aprx{\ast}{M}$ in general, we use the formalism of the split Grothendieck group to express the addition of a negative number of copies of an interval in a direct sum. For more details, see for example the notes \cite[Chapter~2]{lu2013algebraic}.
\begin{definition}
  \label{defn:grotgroup}
  The \emph{split Grothendieck group} $\Gr (\calC)$ of an additive category $\calC$ is the free abelian group generated by isomorphism classes $[C]$ of objects in $\calC$ modulo the relations $[C_1 \oplus C_2] = [C_1] + [C_2]$ for all objects $C_1,C_2$ of $\calC$.
  For an object $C$ of $\calC$, we denote by $\br{C}$ the element of $\Gr(\calC)$ represented by $[C]$.
\end{definition}
In the following we consider the split Grothendieck group $\Gr(\rep\Gf{m,n})$ of $\rep\Gf{m,n}$.
Then by the Krull-Schmidt theorem we easily see that it has a basis
$\{\br{L} \mid L \in \mathcal{L}\}$, where $\mathcal{L}$ is a complete set of representatives
of the isomorphism classes of indecomposable representations of $\Gf{m,n}$ (see \cite[Theorem~2.3.6]{lu2013algebraic}).
Thus each $X \in \Gr(\rep\Gf{m,n})$ is uniquely expressed in the form
\[
X = \sum_{L \in \mathcal{L}} a_L\br{L}
\]
with $a_L \in \mathbb{Z}$ for all $L \in \mathcal{L}$.
Here we define the representations
\begin{equation}\label{eq:pos-neg}
X_+ := \bigoplus_{\substack{L \in \mathcal{L}\\a_L \ge 0}} L^{a_L}\quad \text{and}\quad
X_- := \bigoplus_{\substack{L \in \mathcal{L}\\a_L < 0}} L^{(-a_L)},
\end{equation}
which are called the \emph{positive} part and the \emph{negative} part of $X$, respectively.
Note that they are representations of $\Gf{m,n}$ with the property that
$X = \br{X_+} - \br{X_-}$
because
$
\br{X_+} =
\sum_{\substack{L \in \mathcal{L}\\a_L \ge 0}} {a_L} \br{L}
\ \text{and}\ 
\br{X_-} =
\sum_{\substack{L \in \mathcal{L}\\a_L < 0}} {(-a_L)} \br{L}
$.
%
Therefore, $X$ can be uniquely presented by the pair $(X_+, X_-)$ of representations of $\Gf{m,n}$.

\begin{definition}[interval-decomposable \EDIT{replacement}]
  \label{def:tildeM}
  Let $M \in \rep\Gf{m,n}$.
  Define the \emph{interval-decomposable \EDIT{replacement}} \EDIT{(or \emph{interval-decomposable approximation})}\footnote{See footnote~\ref{footnotereplacement}.}
  $\aprxM{\ast}{M}$ of $M$
  by
  \begin{equation}
    \label{eq:tildeM}
    \aprxM{\ast}{M} := \sum\limits_{I\in\II{m,n}} {\aprx{\ast}{M}(I)}\br{V_I}
    \in \Gr(\rep\Gf{m,n})
  \end{equation}
  for $\ast = \rss, \rcc, \text{or}, \tot$.
\end{definition}
By the above observation, $\aprxM{\ast}{M}$ can be expressed by the pair
$\left(\aprxM{\ast}{M}_+, \aprxM{\ast}{M}_-\right)$ of interval-decomposable representations, where
  \[
    \aprxM{\ast}{M}_+ =
    \bigoplus\limits_{\substack{I\in\II{m,n}\\\aprx{\ast}{M}(I)>0}} {V_I}^{\aprx{\ast}{M}(I)} \text{ and }
    \aprxM{\ast}{M}_- =
    \bigoplus\limits_{\substack{I\in\II{m,n}\\\aprx{\ast}{M}(I)<0}} {V_I}^{(-\aprx{\ast}{M}(I))}.
  \]

\begin{theorem}
  \label{thm:intervaltilde}
  Let $M \in \rep\Gf{m,n}$ be interval-decomposable. Then, $\aprxM{\ast}{M} = \br{M}$,
or equivalently, $\aprxM{\ast}{M}_+ \cong M$ and $\aprxM{\ast}{M}_- = 0$.
\end{theorem}
\begin{proof}
  Because $M$ is interval-decomposable, $\aprx{\ast}{M} = d_M$. The conclusion follows immediately from this.
\end{proof}
Note that the converse trivially holds. If $\aprxM{\ast}{M} = \br{M}$
then 
$M$ is interval-decomposable.



Let us discuss the relationship between $M$ and $\aprxM{\ast}{M}$.
In particular, we focus on dimension vectors and rank invariants. 

\begin{example}[Continuation of Example~\ref{ex:negativedtilde}]
  With the same notation as in Example~\ref{ex:negativedtilde}, we have the equality
  \[
    \begin{array}{rcl}
    \displaystyle\sum\limits_{I \in \II{2,3}} \aprx{\rss}{M} (I) \cdot  \udim (V_{I})
    &=&
    \left(\smat{ 1 & 1 & 1 \\ 0& 1 & 1 }\right) 
    + \left(\smat{ 0 & 1 & 0 \\ 0& 1 & 1 }\right)
    + \left(\smat{ 1 & 1 & 0 \\ 0& 0 & 0 }\right)
    - \left(\smat{ 1 & 1 & 0 \\ 0& 1 & 1 }\right) \\
    &=&\left(\smat{1 & 2 & 1 \\ 0 & 1 & 1 }\right)\\
      &=& \udim (M).
    \end{array}
  \]
\end{example}
For $\aprx{\rcc}{M}$, we have a similar equality of the dimension vectors for the example above. This is not a coincidence, and in fact the equality always holds (see Corollary~\ref{cor:dimvec}). First we prove the following stronger statement.
\begin{theorem} \label{thm:rank}
  Let $M$ be a representation of $\Gf{m,n}=(Q,R)$, and
  let $i$ and $j$ be vertices of $Q$ such that there exists a path from $i$ to $j$ in $Q$.
Then we have 
\begin{equation}
  \label{eq:rank}
  \sum\limits_{I\in \II{m,n}} \aprx{\ast}{M}(I) \cdot \rank V_I (i \to j) = \rank M(i \to j).
\end{equation}
for $\ast = \rss,\rcc,\tot$.
\end{theorem}

To prove the theorem above we need the following lemma, which is the essence of Theorem~\ref{thm:rank}.

\begin{lemma}
  \label{lem:prime}
  Let $M \in \rep\Gf{m,n}$ and $I \in \II{m,n}$. Then
  \[
    \dbar{\ast}{M}{I} = \sum\limits_{I\leq J \in \II{m,n}} \aprx{\ast}{M} (J)
  \]
\end{lemma}
\begin{proof}
  This follows from M\"obius inversion. That is, by definition $\aprx{\ast}{M} := \mu \dbarfun{\ast}{M} $ and thus
  \[
    \dbarfun{\ast}{M} = \zeta \aprx{\ast}{M}
  \]
  since $\mu^{-1} = \zeta$. The right-hand side expanded out gives the result.
\end{proof}




Then we prove Theorem~\ref{thm:rank}.

\begin{proof}[Proof of Theorem~\ref{thm:rank}]


  Since there is a path from $i$ to $j$, the rectangle with source $i$ and sink $j$ exists. We denote this rectangle with source $i$ and sink $j$ by $R_{i,j}$.

  We note that for an interval $I\in\II{m,n}$, $\rank V_I(i \to j)$ is $1$ if and only if $I$ contains the rectangle $R_{i,j}$ and is $0$ otherwise. This gives the first equality in the following computation. We have
  \[    
    \begin{array}{rcl}
      \displaystyle\sum\limits_{I\in \II{m,n}} \aprx{\ast}{M} (I) \cdot \rank V_I (i \to j)
      & = &  \displaystyle\sum\limits_{R_{i,j}\leq I\in \II{m,n}} \aprx{\ast}{M} (I) \\
      & = & \dbar{\ast}{M}{R_{i,j}} \\
      & = & \rank M(i \to j),
    \end{array}
  \]
   where the second equality follows from Lemma~\ref{lem:prime}, and the last equality follows by applying Proposition~\ref{prop:rankinv}.
\end{proof}

As a corollary of Theorem~\ref{thm:rank}, we have the following desired equation for dimension vectors.

\begin{corollary} \label{cor:dimvec}
Let $M$ be a representation of $\Gf{m,n}$.
Then we have 
\begin{equation}
  \label{eq:dimvec}
  \sum\limits_{I\in \II{m,n}} \aprx{\ast}{M}(I) \cdot \udim (V_I) = \udim (M).
\end{equation}
\end{corollary}

\begin{proof}
It is enough to show that for any $i \in G_0$, 
$$
\sum\limits_{I\in \II{m,n}} \aprx{\ast}{M} (I) \cdot (\udim (V_I))_i = (\udim (M))_i.
$$
Note that $(\udim (V_I))_i = \rank V_I(i\to i)$ and $(\udim (M))_i = \rank M(i \to i)$, where the path $i\to i$ means the path $e_i$ of length $0$ at $i$.
Thus, by Theorem~\ref{thm:rank}, we obtain the above equation.
\end{proof}

Let us give another consequence of this result,
which warns us against thinking of \EDIT{$\aprx{\ast}{M}$ as a kind of}
approximation 
in terms of functions. \EDIT{In more detail, each}
$M\in\rep{\Gf{m,n}}$ can be written as $M \cong M_I \oplus X$, where $M_I$ is interval-decomposable, and $0\neq X$ has no interval representation as a summand. By Lemma~\ref{lem:dtilde-decomp},
\begin{equation}
  \label{eq:tilded}
  \aprx{\ast}{M} = \aprx{\ast}{M_I} + \aprx{\ast}{X} = d_{M_I} + \aprx{\ast}{X}  : \II{m,n} \rightarrow \mathbb{R}
\end{equation}
where we also use the fact that $\aprx{\ast}{M_I} = d_{M_I}$ because $M_I$ is interval-decomposable. Restricted to $\II{m,n}$, $d_{M}$ has the same values as $d_{M_I}$. Precisely speaking, by our abuse of notation $d_M : \II{m,n} \rightarrow \mathbb{R}$ above is the full multiplicity function $d_M$ restricted to the set of interval representations, which can be identified with $\II{m,n}$.
Thus, we may be tempted to think of using $\aprx{\ast}{M}$
to approximate
$d_{M_I} = d_M$ as functions on $\II{m,n}$. To measure the error involved, we use the $\ell_1$-norm of functions $f:\II{m,n}\rightarrow \mathbb{R}$ defined by $\left\lVert f \right\rVert_1 = \sum_{I\in\II{m,n}} |f(I)|$. Let us consider the value of
\[
  \left\lVert \aprx{\ast}{X} \right\rVert_1 =  \left\lVert \aprx{\ast}{M} - d_M \right\rVert_1.
\]
We remind the reader that we are considering $d_M$ as a function on $\II{m,n}$ by restriction.
\begin{corollary}
  Let $\Gf{m,n}$ be an equioriented commutative grid of size at least $2\times 5$ or $5\times 2$.
  For any $\ell \in \mathbb{N}$, there exists an indecomposable non-interval representation $X \in \rep\Gf{m,n}$, such that
  \[
    \left\lVert \aprx{\ast}{X} \right\rVert_1 \geq \ell.
  \]
\end{corollary}
\begin{proof}
  The construction in \cite{buchet_et_al:socg} provides such an indecomposable non-interval $X \in \rep\Gf{2,5}$ (for $\Gf{m,n}$ larger than $2\times 5$, we simply pad with zero spaces and zero maps):
  \[
    \begin{tikzcd}[ampersand replacement=\&]
      K^\ell \rar{\left[\smat{E\\0}\right]}
      \&
      K^{2\ell} \rar
      \&
      K^{2\ell} \rar{\left[\smat{E&0}\right]}
      \&
      K^\ell \rar
      \&
      0
      \\
      0 \rar \uar
      \&
      K^\ell \rar{\left[\smat{E\\0}\right]} \uar{\left[\smat{E\\E}\right]}
      \&
      K^{2\ell} \rar \uar{\left[\smat{E&E\\E&J}\right]}
      \&
      K^{2\ell} \rar{\left[\smat{E&0}\right]} \uar{\left[\smat{E&E}\right]}
      \&
      K^\ell \uar        
    \end{tikzcd}
  \]
  where each $E$ is an $\ell\times\ell$ identity matrix, and $J$ is the $\ell \times \ell$ Jordan block with eigenvalue $\lambda=1$.
  
  Let $i$ be one of the vertices such that $X(i)$ has dimension at least $\ell$. We compute:
  \[
    \begin{array}{rcl}
      \ell \leq \dim X(i)
      &=& \displaystyle\sum\limits_{I\in \II{m,n}} \aprx{\ast}{X} (I) \cdot (\udim (V_I))_i \\
      &=& \displaystyle\sum\limits_{I : i \in I} \aprx{\ast}{X} (I) \\
      &\leq& \displaystyle\sum\limits_{I : i \in I} | \aprx{\ast}{X} (I)| \\
      &\leq& \displaystyle\sum\limits_{I \in \II{m,n}} | \aprx{\ast}{X} (I)| \\
      &=& \left\lVert \aprx{\ast}{X} \right\rVert_1,
    \end{array}
  \]
  where the first line follows from Corollary~\ref{cor:dimvec}.
\end{proof}

\begin{remark}
  A simpler proof can be provided, if we allow $X$ to not be indecomposable in the preceding corollary, as follows.
  Let $N$ be an indecomposable non-interval representation, which is known to exist. For example, the above indecomposable can be reused. Then, defining $X$ as the direct sum of $\ell$ copies of $N$, we have that $X$ and
  \[
    \left\lVert \aprx{\ast}{X} \right\rVert_1 = \left\lVert \sum_{i=1}^\ell \aprx{\ast}{N} \right\rVert_1 = \ell \left\lVert \aprx{\ast}{N} \right\rVert_1 \geq \ell
  \]
  since $\left\lVert \aprx{\ast}{N} \right\rVert_1 \geq 1$ (otherwise $\aprx{\ast}{N}=0$ and thus $N=0$, a contradiction).  
\end{remark}

In other words, the ``error term'' $\left\lVert \aprx{\ast}{X} \right\rVert_1$ can be made arbitrarily large by varying $M$. 
  However, in the above analysis, we considered the ``error term''
$\left\lVert \aprx{\ast}{X} \right\rVert_1 =  \left\lVert \aprx{\ast}{M} - d_M \right\rVert_1$
where $d_M$ is considered as a function on $\II{m,n}$ by restriction. That is, its values on non-intervals are ignored. A more comprehensive analysis could potentially take into account those terms as well.

Finally, let us give an interpretation of Theorem~\ref{thm:rank} and Corollary~\ref{cor:dimvec}.
The left-hand side
\[
  \sum\limits_{I\in \II{m,n}} \aprx{\ast}{M}(I) \cdot \rank V_I (i \to j)
\]
of Equation~\eqref{eq:rank} in Theorem~\ref{thm:rank} and the left-hand side
\[
  \sum\limits_{I\in \II{m,n}} \aprx{\ast}{M}(I) \cdot \udim (V_I)
\]
of Equation~\eqref{eq:dimvec} in Corollary~\ref{cor:dimvec} can be viewed as the rank invariant and the dimension vector of the interval-decomposable \EDIT{replacement} 
\[
  \aprxM{\ast}{M} = \sum\limits_{I\in\II{m,n}} {\aprx{\ast}{M}(I)}\br{V_I},
\] respectively. That is, the rank invariant (dimension vector, respectively) of $\aprxM{\ast}{M}$ can be defined by adding the rank invariants (dimension vectors, respectively) of its summands.
With this, Theorem~\ref{thm:rank} and Corollary~\ref{cor:dimvec} simply states that
the interval-decomposable \EDIT{replacement} $\aprxM{\ast}{M}$ preserves the rank invariant and dimension vector of $M$. It is in this sense that we think of \EDIT{replacing (or loosely speaking, approximating)}
$M$ by $\aprxM{\ast}{M}$.


\section{Algorithms for equioriented commutative ladders}
\label{sec:algorithm}

Let $M$ be a persistence module over an equioriented $m \times n$ commutative grid.
For completeness, we first present a high-level overview of an algorithm for the computation of our proposed interval-decomposable \EDIT{replacement} $\aprxM{\ast}{M}$. Afterwards, we consider the case of persistence modules over equioriented commutative ladders ($2\times n$ commutative grids).

The computation of \EDIT{interval-decomposable replacement}
$\aprxM{\ast}{M} = \sum\limits_{I\in\II{m,n}} {\aprx{\ast}{M}(I)}\br{V_I}$ of $M$
involves two major steps:
\begin{enumerate}
\item (Algorithm~\ref{alg:compressedmultiplicity}) computation of the compressed multiplicity function
  $\dbarfun{\ast}{M} : \II{m,n} \rightarrow \mathbb{N}$, defined by
  \[
    \dbar{\ast}{M}{I}:= \nd{\VComp{\ast}{I}{M}}{\VComp{\ast}{I}{V_I}}
  \]
  for $I\in \II{m,n}$, and
\item (Algorithm~\ref{alg:inversion}) computation of the M\"obius inversion
  $\aprx{\ast}{M} = \mu \dbarfun{\ast}{M}$ given by
  \[
    \aprx{\ast}{M}(I):= \dbar{\ast}{M}{I} + \sum\limits_{\emptyset \not= S \subseteq \Cov (I)} (-1)^{\# S} \dbar{\ast}{M}{\bigvee S}.
  \]
  for $I \in \II{m,n}$.
\end{enumerate}

Algorithm~\ref{alg:compressedmultiplicity} below for the computation of the compressed multiplicity simply expands upon the definition.
\begin{algorithm}[H]
  \caption{Compressed multiplicity $\dbarfun{\ast}{M}$ of $M$}
  \label{alg:compressedmultiplicity}
  \begin{algorithmic}[1]
    \Function{CompressedMultiplicity}{$M$}
    \State Initialize the function $\dbarfun{\ast}{M}$ on $\II{m,n}$ \EDIT{to zero}
    \For {$I \in \II{m,n}$}
    \State Compute the compressed representation $M' = \VComp{\ast}{I}{M}$.\label{algline:Mcomp}
    \State Compute the compressed representation $I' = \VComp{\ast}{I}{V_I}$.
    \StatexIndent[3] (which is simply the interval representation with the whole of $\QComp{\ast}{I}$ as support)
    \State Compute the multiplicity $d_{M'}(I')$ of $I'$ in $M'$.\label{algline:mult}
    \State $\dbar{\ast}{M}{I} \gets d_{M'}(I')$
    \EndFor
    \State \Return $\dbarfun{\ast}{M}$
    \EndFunction
  \end{algorithmic}
\end{algorithm}

Line~\ref{algline:Mcomp} of Algorithm~\ref{alg:compressedmultiplicity} for the compressed representation $M' = \VComp{\ast}{I}{M}$ simply means forgetting about the vector spaces (internal linear maps, resp.) of $M$ corresponding to objects (morphisms, resp.) \emph{not} in the compressed category $\QComp{\ast}{I}$.
Note that depending on how $M$ is stored, extra computations are needed
(if some of the internal maps of $M$ are not explicitly stored,
they may need to be computed explicitly and stored if they rely on internal maps about to be forgotten).
We provide an example of this with the $2\times n$ case later.

In general, the computation of the multiplicity $d_{M'}(I')$ of $I'$ in $M'$
(Line~\ref{algline:mult} of Algorithm~\ref{alg:compressedmultiplicity})
can be accomplished by computing the dimensions of certain homomorphism spaces to entries in the almost split
sequence\footnote{
A non-split short exact sequence $(E):\ 0 \to X \xrightarrow{f} Y \xrightarrow{g} Z \to 0$
is called an {\em almost split sequence} starting at $X$
if both $X$ and $Z$ are indecomposable, and if for any homomorphism
$h \colon X \to V$, either $h$ is a split monomorphism or
the pushout of $(E)$ along $h$ splits.
} 
starting at $I'$ (see \cite[Theorem~3]{Asashiba2017}, \cite[Corollary.~2.3]{dowbor2007multiplicity} and
also \cite[Algorithms~3,~4]{asashiba2022interval}).
\EDIT{Indeed, for Algorithm~\ref{alg:compressedmultiplicity} and
its specialization to the $2\times n$ commutative grids in Algorithm~\ref{alg:compressedmultiplicity2n},
we rely heavily on \cite{Asashiba2017}.}

Algorithm~\ref{alg:inversion} is also a straightforward expansion of the definition.

\begin{algorithm}[H]
  \caption{M\"obius inversion $\aprx{\ast}{M}$ of $\dbarfun{\ast}{M}$}
  \label{alg:inversion}
  \begin{algorithmic}[1]
    \Function{M\"obiusInversion}{$\dbarfun{\ast}{M}$}
    \State Initialize the function $\aprx{\ast}{M}$ on $\II{m,n}$ \EDIT{to zero}
    \For {$I \in \II{m,n}$}
      \State $a \gets \dbarfun{\ast}{M}(I)$
      \State Compute $\Cov (I)$ \label{algline:cov}
      \For {$\emptyset \neq S \subseteq \Cov (I)$}
        \State Compute $\bigvee S$
        \State $a \gets a + (-1)^{\# S} \dbarfun{\ast}{M}(\bigvee S)$
      \EndFor
      \State $\aprx{\ast}{M}(I) \gets a$
    \EndFor
    \State \Return $\aprx{\ast}{M}$
    \EndFunction
  \end{algorithmic}
\end{algorithm}

Algorithm~\ref{alg:inversion} requires the computation of joins of cover elements of $I$. We comment on this below.
Let $I  = \bigsqcup_{i=s}^t [b_i,d_i]_i$.
By Proposition~\ref{prop:cover}, the elements of $\Cov (I)$ are given by a specific form.
We recall that Proposition~\ref{prop:cover} only provides a list of candidates,
from which picking up all valid intervals forms $\Cov(I)$.
We single out the following four \emph{potential} cover elements specified by Proposition~\ref{prop:cover}
that need special consideration:
\begin{enumerate}
\item extension of the 
  \EDIT{top row}
  of $I$ by one adjacent vertex left of the row (top-left)
  \[
    C_{tl} =
    \displaystyle\bigsqcup_{i=s}^t [b'_i,d_i]_i
    \text{, where }
    b'_i = \begin{cases}
      b_i-1 & \text{if } i=t,\\
      b_i & \text{otherwise,}
    \end{cases}
  \]
\item extension of the 
  \EDIT{bottom row}
  of $I$ by one adjacent vertex right of the row (bottom-right)
  \[
    C_{br} =
    \displaystyle\bigsqcup_{i=s}^t [b_i,d'_i]_i
    \text{, where }
    d'_i = \begin{cases}
      d_i+1 & \text{if } i=s,\\
      d_i & \text{otherwise,}
    \end{cases}
  \]
\item addition of one vertex above the upper-left vertex of $I$ (top)
  \[
    C_{t} = \displaystyle\bigsqcup_{i=s}^t [b_i,d_i]_i \sqcup [b_t,b_t]_{t+1},
  \]
\item addition of one vertex below the lower-right vertex of $I$ (bottom)
  \[
    C_{b} = \displaystyle [d_s,d_s]_{s-1} \sqcup \bigsqcup_{i=s}^t [b_i,d_i]_i.
  \]
\end{enumerate}

\begin{remark}
\label{remark:hensyuu}
It is clear that if $S \subset \Cov(I)$
\begin{itemize}
\item does not contain both $C_{tl}$ and $C_{t}$, and
\item does not contain both $C_{br}$ and $C_b$,
\end{itemize}
then $\bigvee S = \bigcup_{C\in S} C$.
That is, simply taking the union is enough since the union is an interval.

Otherwise, we need to add at most two vertices to $\bigcup_{C\in S} C$ in order to obtain $\bigvee S$.
If $S \subset \Cov(I)$ contains both $C_{tl}$ and $C_{t}$, then an additional vertex in the top left needs to be added to form an interval. Similarly, if $S \subset \Cov(I)$ contains both $C_{br}$ and $C_{b}$, then an additional vertex in the bottom right needs to be added to form an interval.
\end{remark}

\begin{example}
  \label{example:cover2}
  We provide an example using the interval $I$ in the commutative grid $\Gf{5,6}$ with candidate vertices marked as in Example~\ref{example:cover}.
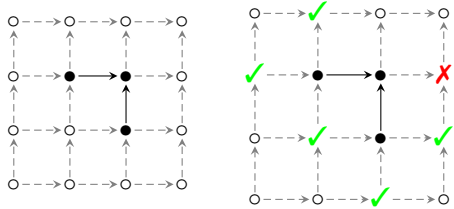
  \[
    \newcommand{\bb}{\bullet}
    \newcommand{\cb}{{\color{green}\text{\cmark}}}
    \newcommand{\xb}{{\color{red}\text{\xmark}}}
    \newcommand{\rard}{\rar[dashed,gray]}
    \newcommand{\uard}{\uar[dashed,gray]}
    \begin{tikzcd}[graphstyle,every matrix/.append style={name=m}]
      \circ \rard & \circ \rard & \circ \rard & \circ \\
      \circ \rard\uard & \bb \rar\uard & \bb \rard\uard & \circ \uard  \\
      \circ \rard\uard & \circ \rard\uard & \bb \rard\uar & \circ \uard \\
      \circ \rard\uard & \circ \rard\uard & \circ \rard\uard & \circ \uard
    \end{tikzcd}
    \quad\quad
    \begin{tikzcd}[graphstyle,every matrix/.append style={name=m}]
      \circ \rard & \cb \rard & \circ \rard & \circ \\
      \cb \rard\uard & \bb \rar\uard & \bb \rard\uard & \xb \uard  \\
      \circ \rard\uard & \cb \rard\uard & \bb \rard\uar & \cb \uard \\
      \circ \rard\uard & \circ \rard\uard & \cb \rard\uard & \circ \uard
    \end{tikzcd}
  \]
  In dimension vector notation,
  \[
    I =
    \left(
      \smat{
        0&0&0&0 \\
        0&1&1&0 \\
        0&0&1&0 \\
        0&0&0&0
      }\right)
  \]
  and all the cover elements are given by
  \begin{align*}
    C_{tl} & =
    \left(
      \smat{
        0&0&0&0 \\
        1&1&1&0 \\
        0&0&1&0 \\
        0&0&0&0
      }\right),
    C_{t}=
    \left(
      \smat{
        0&1&0&0 \\
        0&1&1&0 \\
        0&0&1&0 \\
        0&0&0&0
      }\right),\\
   C_{br} & =
    \left(
      \smat{
        0&0&0&0 \\
        0&1&1&0 \\
        0&0&1&1 \\
        0&0&0&0
        }\right),
    C_{b}=
    \left(
      \smat{
        0&0&0&0 \\
        0&1&1&0 \\
        0&0&1&0 \\
        0&0&1&0
      }\right),\\
   C & =
       \left(
       \smat{
       0&0&0&0 \\
       0&1&1&0 \\
       0&1&1&0 \\
       0&0&0&0
     }\right).
  \end{align*}
  Thus, for example,
  \[
    C_{br} \vee C_t \vee C
    =
    \left(
      \smat{
        0&1&0&0 \\
        0&1&1&0 \\
        0&1&1&1 \\
        0&0&0&0
      }\right)
    = C_{br} \cup C_t \cup C
  \]
  while
  \[
    C_{tl} \vee C_t \vee C_{br}
    =
    \left(
      \smat{
        1&1&0&0 \\
        1&1&1&0 \\
        0&0&1&1 \\
        0&0&0&0
      }\right)
    = \{v\} \cup C_{tl} \cup C_t \cup C_{br}
  \]
  where $v$ is the vertex at the upper-left corner.
\end{example}

\begin{theorem}
  \label{thm:algcost_inversion}
  Algorithm~\ref{alg:inversion}, which computes $\aprx{\ast}{M}$ given $\dbarfun{\ast}{M}$, can be performed with time complexity $O(\#\II{m,n}2^DD m)$, where $D = \max_{I \in \II{m,n}} \# \Cov (I)$.
\end{theorem}
\begin{proof}
  For each $I \in \II{m,n}$, there are at most $2^D -1$ nonempty subsets $S$ of $\Cov (I)$.
  By the formula
  \[
    \aprx{\ast}{M}(I):= \dbar{\ast}{M}{I} + \sum\limits_{\emptyset \not= S \subseteq \Cov (I)} (-1)^{\# S} \dbar{\ast}{M}{\bigvee S}
  \]
  for each $S$, we need to compute $\bigvee S$, which is the join of at most $D$ intervals.

  We first compute $\bigcup_{C \in S} C$ by the following.
  Assuming that intervals are represented in the form of $I  = \bigsqcup_{i=s}^t [b_i,d_i]_i$ (row-wise), with the number of rows equal to $m$, 
  the union of two cover elements can be computed by iterating through the $m$ rows and taking the union of the corresponding intervals $[b_i,d_i]_i \cup [b'_i,d'_i]_i$. We iterate over the elements of $S$ (at most $D$) to obtain $\bigcup_{C \in S} C$. 
  
  Finally, the above discussion around Remark~\ref{remark:hensyuu} concerning the four cover elements $C_{tl}, C_t, C_{br}, C_b$ that need special consideration provides the computation of $\bigvee S$ 
  by modifying the union $\bigcup_{C \in S} C$.
  We simply need to check for the presence of both $C_{tl}$ and $C_{t}$ in $S$, and both $C_{br}$ and $C_b$ in $S$, and add the additional vertices to $\bigcup_{C \in S} C$ to obtain $\bigvee S$, as noted in Remark~\ref{remark:hensyuu}.
  
  By the above, we have as an upper bound $\#\II{m,n} \cdot (2^D-1) \cdot D \cdot m$ operations, giving the claimed time complexity. 
\end{proof}

Next, we consider the case of equioriented commutative ladders with $\ast = \rss$,
where it has been noted in Subsection~\ref{subsec:2byn} that the ss-compressed category is of
Dynkin $A_n$-type with $n \leq 4$ (Proposition~\ref{prop:ss-fin}).
So, let $M$ be a persistence module over the $2\times n$ commutative grid, and let
\[
  d = \max_{v \in \left(\Gf{2,n}\right)_0} \dim M(v).
\]
In particular $M \in \rep{\Gf{2,n}}$ is given as the following collection of vector spaces and linear maps
\[
  \begin{tikzcd}[column sep=6em]
    M(2,1) \rar{\scriptsize{M((2,1)\rightarrow(2,2))}} &
    M(2,2) \rar{\scriptsize{M((2,1)\rightarrow(2,3))}} &
    \cdots \rar{\scriptsize{M((2,n-1)\rightarrow(2,n))}} &
    M(2,n)\\
    M(1,1) \rar{\scriptsize{M((1,1)\rightarrow(1,2))}}
    \uar[swap]{\scriptsize{M((1,1)\rightarrow(2,1))}}
    &
    M(1,2) \rar{\scriptsize{M((1,2)\rightarrow(1,3))}}
    \uar[swap]{\scriptsize{M((1,2)\rightarrow(2,2))}}
    &
    \cdots \rar{\scriptsize{M((1,n-1)\rightarrow(1,n))}}
    &
    M(1,n)
    \uar{\scriptsize{M((1,n)\rightarrow(2,n))}}
  \end{tikzcd}
\]
such that
\[
  \begin{tikzcd}[column sep=6em]
    M(2,j) \rar{\scriptsize{M((2,j)\rightarrow(2,j+1))}} &
    M(2,j+1) \\
    M(1,j) \rar{\scriptsize{M((1,j)\rightarrow(1,j+1))}}
    \uar{\scriptsize{M((1,j)\rightarrow(2,j))}}
    &
    M(1,j+1)
    \uar[swap]{\scriptsize{M((1,j+1)\rightarrow(2,j+1))}}
  \end{tikzcd}
\]
commutes for all $j \in \{1,2,\hdots,n-1\}$.
For $(x,y), (i,j)$ distinct vertices of $\Gf{2,n}$ such that $x \leq i$ and $y \leq j$
(that is, there exists a path from $(x,y)$ to $(i,j)$ in $\Gf{2,n}$),
$M((x,y) \rightarrow (i,j))$ is the composition
$M(p) = M(\alpha_\ell) \cdots  M(\alpha_1)$ where $p = (\alpha_l,\hdots,\alpha_\ell)$
is a path from $(x,y)$ to $(i,j)$ in $\Gf{2,n}$. Note that by the commutativity relations, the composition does not depend on the path chosen.

In Algorithm~\ref{alg:compressedmultiplicity2n},
we specialize Algorithm~\ref{alg:compressedmultiplicity} to this setting and add more details.
In particular, we precompute all the compositions $M((x,y) \rightarrow (i,j))$
(as each will be used at some point in the algorithm, anyway), and explicitly write down formulae for
$d_{M'}(I')$ using ranks of certain matrices.

\begin{algorithm}[H]
  \caption{$\rss$-compressed multiplicity ($2\times n$ case)}
  \label{alg:compressedmultiplicity2n}
  \begin{algorithmic}[1]
    \Function{ssCompressedMultiplicityTwoByN}{$M$}
    \State Initialize the function $\dbarfun{\rss}{M}$ on $\II{2,n}$ \EDIT{to zero}
    \State Compute $M((x,y) \rightarrow (i,j))$ for all $(x,y) \neq (i,j)$ with a path from $(x,y)$ to $(i,j)$
    \label{algline:precompmatrices}
    \For {$I \in \II{2,n}$} \label{algline:FORcompressSTART}
    \State Compute $d_{M'}(I')$ using the formula in Proposition~\ref{prop:2nmult},
    \StatexIndent[3] where $M' = \VComp{\rss}{I}{M}$ and $I' = \VComp{\rss}{I}{V_I}$.
    \State $\dbar{\rss}{M}{I} \gets d_{M'}(I')$
    \EndFor \label{algline:FORcompressEND}
    \State \Return $\dbarfun{\rss}{M}$
    \EndFunction
  \end{algorithmic}
\end{algorithm}

\begin{proposition}
  \label{prop:2nmult}
  Let $M \in \rep K\Gf{2,n}$, $I \in \II{2,n}$ and
  let $M' = \VComp{\rss}{I}{M}$ and $I' = \VComp{\rss}{I}{V_I}$
  be their respective compressed representations of $\QComp{\rss}{I}$.
  Below, we use the convention that 
  the symbols
  $s_1$ and $t_1$ 
  stand for vertices
  on row $1$ (i.e. have coordinates $(1,?)$),
  and that $s_2$ and $t_2$ 
  stand 
  for vertices on row $2$ (i.e. have coordinates $(2,?)$).

  Then $I$ is in one of the following four cases, and the value of the compressed multiplicity
  $\dbar{\rss}{M}{I} = d_{M'}(I')$ is given by the respective formula.
  \begin{itemize}
  \item If {$I$ is a rectangle with source $s$ and sink $t$}
    then
    \[
      d_{M'}(I') = \rank M(s\rightarrow t)
    \]

  \item If {$I$ has sources $s_1,s_2$ and sink $t_2$}
    then
    \begin{align*}
      d_{M'}(I') &=
                   \rank M(s_2 \rightarrow t_2)
                   + \rank M(s_1 \rightarrow t_2) \\
                 & \quad - \rank
                   \begin{bmatrix}
                     M(s_2 \rightarrow t_2) &
                     M(s_1 \rightarrow t_2)
                   \end{bmatrix}
    \end{align*}

  \item If {$I$ has source $s_1$ and sinks $t_1$, $t_2$}
    then
    \[
      d_{M'}(I') =
      \rank M(s_1 \rightarrow t_2)
      + \rank M(s_1 \rightarrow t_1)
      - \rank
      \begin{bmatrix}
        M(s_1 \rightarrow t_2) \\
        M(s_1 \rightarrow t_1)
      \end{bmatrix}
    \]

  \item If {$I$ has sources $s_1,s_2$ and sinks $t_1$, $t_2$}
    then
    \begin{align*}
      d_{M'}(I') & =
                   \rank
                   \begin{bmatrix}
                     M(s_2 \rightarrow t_2) & M(s_1 \rightarrow t_2) \\
                     0  & M(s_1 \rightarrow t_1) \\
                   \end{bmatrix}
      + \rank M(s_1 \rightarrow t_2)  \\
      & \quad - \rank
      \begin{bmatrix}
        M(s_1 \rightarrow t_2) \\
        M(s_1 \rightarrow t_1)
      \end{bmatrix}
      - \rank
      \begin{bmatrix}
        M(s_2 \rightarrow t_2) &
        M(s_1 \rightarrow t_2)
      \end{bmatrix}
    \end{align*}
  \end{itemize}
\end{proposition}
\begin{proof}
  Each element $I$ of $\II{2,n}$ has a staircase form, which is denoted by:
  \[
    I  = \bigsqcup_{i=j}^k [b_i,d_i]_i
  \]
  for some integers $1\leq j\leq k \leq 2$ and some integers $1 \leq b_i \leq d_i \leq n$ for
  each $j \leq i \leq k$ such that $b_{i+1}\leq b_{i}\leq d_{i+1}\leq d_{i}$ for any $i\in \{j,\dots,k-1\}$.

  The two cases given by
  \begin{itemize}
  \item $j = k$, or
  \item $b_1 = b_2$ and $d_1 =d_2$
  \end{itemize}
  correspond to $I$ being a rectangle (with source $s = (j,b_j)$ and sink $t = (j,d_j)$, or
  source $s = (1,b_1)$ and sink $t = (2,d_2)$, respectively).
  Here, Proposition~\ref{prop:rankinv} gives the formula for the compressed multiplicity.

  Thus, we are left with the cases that $1 = j < k = 2$, and that $b_1 \neq b_2$ or $d_1 \neq d_2$.
  By the general restriction that $b_{2}\leq b_{1}\leq d_{2}\leq d_{1}$, we have the following three cases
  \begin{itemize}
  \item $b_2 < b_1 \leq d_2 = d_1$.
    This corresponds to the case that $I$ has
    sources $s_1 = (1,b_1), s_2 = (2,b_2)$ and
    sink $t_2 = (2,d_2)$, as illustrated below:
    \[
      \tikzset{main node/.style={circle,fill=black, inner sep=1pt},}
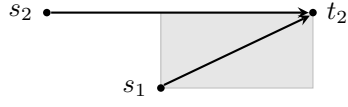
      \begin{tikzpicture}

        \draw[fill=gray!20,draw=gray!50] (0,0) -- (2,0) -- (2,1) -- (0,1) -- cycle;
        \node[main node] at (-1.5,1) (S2) {};
        \node[main node] at (2,1) (T2) {};
        \node[main node] at (0,0) (S1) {};

        \node[left=0 pt of S1.west] {$s_1$};
        \node[left=0 pt of S2.west] {$s_2$};
        \node[right=0 pt of T2.east] {$t_2$};

        \draw[-stealth, thick] (S1.north east) -- (T2.south west);
        \draw[-stealth, thick] (S2.east) -- (T2.west);
      \end{tikzpicture}
    \]
    with $\QComp{\rss}{I}:
    \begin{tikzcd}[graphstyle]
      s_1 \rar & t_2 & s_2 \lar
    \end{tikzcd}
    $ emphasized. Then, the compressed representations are given by
    \[
      I':
      \begin{tikzcd}
        K \rar{1} & K & K \lar[swap]{1}
      \end{tikzcd}
      \text{ and }
      M':
      \begin{tikzcd}
        M(s_1) \rar{M(s_1\rightarrow t_2)} & M(t_2) & M(s_2) \lar[swap]{M(s_2 \rightarrow t_2)}
        \mathrlap{.}
      \end{tikzcd}
    \]
    We note that $I'$ is injective with socle given by
    \[
    \soc I':
      \begin{tikzcd}
        0 \rar{0} & K & 0\mathrlap{.} \lar[swap]{0}
      \end{tikzcd}
    \]
    Using \cite[Theorem~3]{Asashiba2017}, we have
    \[
      d_{M'}(I') = \dim \Hom (I',M') - \dim \Hom(I'/\soc I', M').
    \]

    A homomorphism $I' \rightarrow M'$ is given by triples $(x,y,z)$ such that
    \[
      \begin{tikzcd}[column sep=4em]
        K \rar{1}\dar{x} & K\dar{y} & K \lar[swap]{1}\dar{z} \\
        M(s_1) \rar{M(s_1\rightarrow t_2)} & M(t_2) & M(s_2) \lar[swap]{M(s_2 \rightarrow t_2)}
      \end{tikzcd}
    \]
    commutes. That is,
    $
    y = M(s_2 \rightarrow t_2) z = M(s_1 \rightarrow t_2) x.
    $
    In other words, the homomorphism space $\Hom (I',M') $ is given by solutions to
    \[
      \begin{bmatrix}
        M(s_2 \rightarrow t_2) &
        M(s_1 \rightarrow t_2)
      \end{bmatrix}
      \begin{bmatrix}
        z \\
        -x
      \end{bmatrix}
      = 0
    \]
    (with $y$ fully determined by $x$),
    which has dimension
    equal to
    \[
    \dim M(s_2) + \dim M(s_1)
    - \rank
    \begin{bmatrix}
      M(s_2 \rightarrow t_2) &
      M(s_1 \rightarrow t_2)
    \end{bmatrix}.
    \]

    On the other hand, a homomorphism $I'/\soc I' \rightarrow M'$ is given by triples $(x,0,z)$ such that
    \[
      \begin{tikzcd}[column sep=4em]
        K \rar{0}\dar{x} & 0\dar{0} & K \lar[swap]{0}\dar{z} \\
        M(s_1) \rar{M(s_1\rightarrow t_2)} & M(t_2) & M(s_2) \lar[swap]{M(s_2 \rightarrow t_2)}
      \end{tikzcd}
    \]
    commutes. Thus
    \begin{align*}
      \dim \Hom(I'/\soc I', M')
      & = \left(\dim M(s_1) - \rank M(s_1 \rightarrow t_2)\right) \\
      & \quad + \left(\dim M(s_2) - \rank M(s_2 \rightarrow t_2)\right).
    \end{align*}
    Combining the above formulas yields the claimed formula for $d_{M'}(I')$.

  \item $b_2 = b_1 \leq d_2 < d_1$.
    This corresponds to the case that $I$ has
    source $s_1 = (1,b_1)$ and
    sinks $t_1 = (1,d_1), t_2 = (2,d_2)$ as illustrated below:
    \[
      \tikzset{main node/.style={circle,fill=black, inner sep=1pt},}
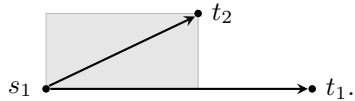
      \begin{tikzpicture}

        \draw[fill=gray!20,draw=gray!50] (0,0) -- (2,0) -- (2,1) -- (0,1) -- cycle;
        \node[main node] at (2,1) (T2) {};
        \node[main node] at (0,0) (S1) {};
        \node[main node] at (3.5,0) (T1) {};

        \node[left=0 pt of S1.west] {$s_1$};
        \node[right=0 pt of T1.east] {$t_1$.};
        \node[right=0 pt of T2.east] {$t_2$};

        \draw[-stealth, thick] (S1.north east) -- (T2.south west);
        \draw[-stealth, thick] (S1.east) -- (T1.west);
      \end{tikzpicture}
    \]
    The proof for the formula of $d_{M'}(I')$ in this case is dual to the previous case.

  \item $b_2 < b_1 \leq d_2 < d_1$.
    This corresponds to the case that $I$ has
    sources $s_1 = (1,b_1), s_2 = (2,b_2)$ and
    sinks $t_1 = (1,d_1), t_2 = (2,d_2)$ as illustrated below:
    \[
      \tikzset{main node/.style={circle,fill=black, inner sep=1pt},}
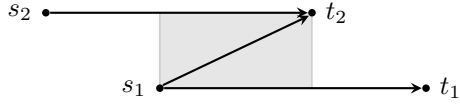
      \begin{tikzpicture}

        \draw[fill=gray!20,draw=gray!50] (0,0) -- (2,0) -- (2,1) -- (0,1) -- cycle;
        \node[main node] at (-1.5,1) (S2) {};
        \node[main node] at (2,1) (T2) {};
        \node[main node] at (0,0) (S1) {};
        \node[main node] at (3.5,0) (T1) {};

        \node[left=0 pt of S1.west] {$s_1$};
        \node[left=0 pt of S2.west] {$s_2$};
        \node[right=0 pt of T1.east] {$t_1$};
        \node[right=0 pt of T2.east] {$t_2$};

        \draw[-stealth, thick] (S1.north east) -- (T2.south west);
        \draw[-stealth, thick] (S2.east) -- (T2.west);
        \draw[-stealth, thick] (S1.east) -- (T1.west);
      \end{tikzpicture}
    \]

    Then, the compressed representations are given by
    \[
      I':
      \begin{tikzcd}
        K \rar{1} & K & K \lar[swap]{1} \rar{1} & K
      \end{tikzcd}
    \]
    and
    \[
      M':
      \begin{tikzcd}[column sep=4em]
        M(s_2) \rar{M(s_2\rightarrow t_2)} & M(t_2) & M(s_1) \lar[swap]{M(s_1 \rightarrow t_2)} \rar{M(s_1 \rightarrow t_1)} & M(t_1)
        \mathrlap{.}
      \end{tikzcd}
    \]
    The almost split sequence starting from $I'$ is given by
    \[
      \begin{tikzcd}
        0 \rar & I' \rar & B \rar & C \rar & 0
      \end{tikzcd}
    \]
    where
    \[
      B:
      \begin{tikzcd}[ampersand replacement=\&]
        K \rar{1} \& K \& K^2 \lar[swap]{\begin{bmatrix}1 & 0\end{bmatrix}} \rar{\begin{bmatrix}0 & 1 \end{bmatrix}} \& K
      \end{tikzcd}
    \]
    and
    \[
      C:
      \begin{tikzcd}[ampersand replacement=\&]
        0 \rar{0} \& 0 \& K \lar[swap]{0} \rar{0} \& 0
      \end{tikzcd}.
    \]
    Using \cite[Theorem~3]{Asashiba2017}, we have
    \begin{equation}
      \label{eq:dimhomformula}
      d_{M'}(I') = \dim \Hom (I',M') - \dim \Hom(B, M') + \dim \Hom(C,M').
    \end{equation}

  For $(x,y,z,w) \in \Hom(I',M')$, the commutativity of
  \[
    \begin{tikzcd}[column sep=4em]
      K \rar{1} \dar{x} & K \dar{y} & K \dar{z} \lar[swap]{1} \rar{1} & K \dar{w} \\
      M(s_2) \rar{M(s_2\rightarrow t_2)} & M(t_2) & M(s_1) \lar[swap]{M(s_1 \rightarrow t_2)} \rar{M(s_1 \rightarrow t_1)} & M(t_1)
    \end{tikzcd}
  \]
  is equivalent to
  \[
    M(s_2 \rightarrow t_2) x = y = M(s_1\rightarrow t_2) z \text{ and } w = M(s_1 \rightarrow t_1) z.
  \]
  Then, each homomorphism is uniquely determined by a solution of
  \[
    \begin{bmatrix}
      M(s_2 \rightarrow t_2) &
      M(s_1\rightarrow t_2)
    \end{bmatrix}
    \begin{bmatrix}
      x \\
      -z
    \end{bmatrix}
    = 0.
  \]
  Thus,
  \begin{equation}
    \label{eq:dimhominterval}
    \begin{split}
      \dim \Hom(I', M') & = \dim M(s_2) + \dim M(s_1) \\
      & \quad - \rank
      \begin{bmatrix}
        M(s_2 \rightarrow t_2) &
        M(s_1 \rightarrow t_2)
      \end{bmatrix}.
    \end{split}
  \end{equation}

  Next, for $(x,y,z,w) \in \Hom(B,M')$, the commutativity of
  \[
    \begin{tikzcd}[column sep=4em,ampersand replacement=\&]
      K \rar{1}\dar{x} \& K\dar{y} \& K^2 \lar[swap]{\begin{bmatrix}1 & 0\end{bmatrix}} \rar{\begin{bmatrix}0 & 1 \end{bmatrix}} \dar{\begin{bmatrix}  z_1 & z_2 \end{bmatrix}}\& K \dar{w} \\
      M(s_2) \rar{M(s_2\rightarrow t_2)} \& M(t_2) \& M(s_1) \lar[swap]{M(s_1 \rightarrow t_2)} \rar{M(s_1 \rightarrow t_1)} \& M(t_1)
    \end{tikzcd}
  \]
  is equivalent to
  \begin{align*}
    M(s_2 \rightarrow t_2) x & = y \\
    M(s_1 \rightarrow t_2) z_1 & = y \\
    M(s_1 \rightarrow t_2) z_2 & = 0 \\
    M(s_1 \rightarrow t_1) z_1 & = 0 \\
    M(s_1 \rightarrow t_1) z_2 & = w\mathrlap{.}
  \end{align*}
  Rewriting the above, we get that each homomorphism is uniquely determined by a solution to
  \[
    \begin{bmatrix}
      M(s_2 \rightarrow t_2) & M(s_1 \rightarrow t_2) \\
      0 & M(s_1 \rightarrow t_1)
    \end{bmatrix}
    \begin{bmatrix}
      x \\
      -z_1
    \end{bmatrix} = 0
    \text{ and }
    M(s_1 \rightarrow t_2) z_2  = 0
  \]
  with $y$ and $w$ determined from $x$ and $z_2$, respectively. Thus,

  \begin{equation}
    \label{eq:dimhommid}
    \begin{split}
      \dim \Hom(B,M') &= \dim M(s_2) + \dim M(s_1) \\
      & \quad - \rank
      \begin{bmatrix}
        M(s_2 \rightarrow t_2) & M(s_1 \rightarrow t_2) \\
        0 & M(s_1 \rightarrow t_1)
      \end{bmatrix}\\
      & \quad + \dim M(s_1)
      - \rank
      M(s_1 \rightarrow t_2)
    \end{split}
  \end{equation}

  Finally, it is clear that
  \begin{equation}
    \label{eq:dimhomtauinv}
    \begin{split}
      \dim \Hom(C,M')
      & = \dim\left(\ker M(s_1 \rightarrow t_1) \cap \ker M(s_1 \rightarrow t_2)\right)\\
      & = \dim \ker \begin{bmatrix}
        M(s_1 \rightarrow t_1) \\
        M(s_1 \rightarrow t_2)
      \end{bmatrix}\\
      & = \dim M(s_1) -
      \rank
      \begin{bmatrix}
        M(s_1 \rightarrow t_1) \\
        M(s_1 \rightarrow t_2)
      \end{bmatrix}.
    \end{split}
  \end{equation}

  Substituting Equations~\eqref{eq:dimhominterval},~\eqref{eq:dimhommid},~\eqref{eq:dimhomtauinv} into Equation~\eqref{eq:dimhomformula} gives the claimed formula.
\end{itemize}
\end{proof}

Let $\omega < 2.373$ be the matrix multiplication exponent \cite{don1990matrix,williams2012multiplying}.
\begin{theorem}[Compressed multiplicity ($2\times n$ case)]
  \label{thm:algcost_compressedmultiplicity2n}
  For $M$ a persistence module over $\Gf{2,n}$, 
  Algorithm~\ref{alg:compressedmultiplicity2n} computes $\dbarfun{\rss}{M}$
  with time complexity
  \[
    O\left(\frac{2^\omega+5}{24}n^4d^\omega\right).
  \]
  where
  $d = \max_{v \in \left(\Gf{2,n}\right)_0} \dim M(v)$.
\end{theorem}

\begin{proof}
  First, let us analyze Line~\ref{algline:precompmatrices} of Algorithm~\ref{alg:compressedmultiplicity2n},
  which computes $M((x,y)\to (i,j)) = M(p)$ for $(x,y) \neq (i,j)$ with a path $p$ from $(x,y)$ to $(i,j)$.
  The value of $M(p)$ for paths $p$ with length equal to $1$ (arrows) are already known.
  Assume that the values of $M(p)$ for all paths of length $\ell$ are already computed.
  Then, the value of $M(p)$ for each path $p$ of length $\ell+1$ can be computed by one matrix multiplication each.
  We note further that $M((x,y)\to (i,j)) = M(p)$ does not depend on which particular path $p$ is taken from $(x,y)$ to $(i,j)$.
  Thus, we can inductively compute the value of $M((x,y)\to (i,j))$ using one matrix multiplication for each pair of vertices $(x,y)$, $(i,j)$ such that $(x,y) \neq (i,j)$ and there is a path of length greater than $1$ from $(x,y)$ to $(i,j)$.
  Since there are
  \[
    \frac{3}{2}(n+1)n - 2n - (3n-2) = O\left(\frac{3}{2}n^2\right)
  \]
  such pairs of vertices $(x,y) \neq (i,j)$ in the $2\times n$ commutative grid
  by a simple combinatorial argument,
  Line~\ref{algline:precompmatrices} of Algorithm~\ref{alg:compressedmultiplicity2n} can be performed in
  $O(\frac{3}{2}n^2d^\omega)$.

  Next, we analyze  Lines~\ref{algline:FORcompressSTART}~to~\ref{algline:FORcompressEND} of Algorithm~\ref{alg:compressedmultiplicity2n}.
  By \cite[Corollary~4.12]{asashiba2022interval},
  there are
  \[
    \#\II{2,n} = \frac{1}{24}n(n+1)(n^2+5n+30) = O\left(\frac{1}{24}n^4\right)
  \]
  intervals $I$ to process. For each interval $I$,
  the computation of $d_{M'}(I')$ using Proposition~\ref{prop:2nmult}
  involves computing the rank of a $2d \times 2d$, a $2d \times d$, a $d \times 2d$, and a $d\times d$ matrix in the worst case.
  Note that the rank of an $e\times f$ matrix ($e \leq f$) can be computed with $O(fe^{\omega-1})$ field operations by Gaussian elimination \cite{ibarra1982generalization}.
  Thus, we get a cost of $O((2^\omega d^\omega +5d^\omega)\frac{1}{24}n^4)$ for the computation of $d_{M'}(I')$.

  Overall, we get a cost of $O(\frac{3}{2}n^2d^\omega + \frac{2^\omega+5}{24}n^4d^\omega)$ dominated by the latter term,
  giving the result.
\end{proof}

For $I \in \II{2,n}$, as shown in Example~\ref{example:cover2n}, $\#\Cov (I) \leq 4$. Thus, we get the following.
\begin{corollary}[M\"obius inversion $\aprx{\ast}{M}$ ($2\times n$ case)]
  \label{cor:algcost_inversion2n}
  With $m=2$, Algorithm~\ref{alg:inversion} (M\"obius inversion $\aprx{\ast}{M}$ of $\dbarfun{\ast}{M}$) can be performed with time complexity
  \[
    O\left(\frac{16}{3}n^4\right).
  \]
\end{corollary}
\begin{proof}
  Substituting $m=2$, $C = 4$, and $\#\II{2,n} = O\left(\frac{1}{24}n^4\right)$ into
  \[
    O(\#\II{m,n}2^C C \min\{m,n\})
  \]
  from Theorem~\ref{thm:algcost_inversion}, we get the result.
\end{proof}

Combining Theorem~\ref{thm:algcost_compressedmultiplicity2n} and Corollary~\ref{cor:algcost_inversion2n} with $\ast = \rss$,
we get an overall cost of
\[
  O\left(\frac{2^\omega+5}{24}n^4d^\omega + \frac{16}{3}n^4\right)
\]
for computing the interval-decomposable \EDIT{replacement} $\aprxM{\rss}{M}$ of $M$ in the $2\times n$ case.

\paragraph{Implementation}
As part of the software ``pmgap'' \cite{pmgap},
we provide an implementation of Algorithms~\ref{alg:compressedmultiplicity2n}~and~\ref{alg:inversion}
in the $2\times n$ case. The software ``pmgap'' builds upon the GAP~\cite{GAP4} package QPA~\cite{qpa},
which provides data structures and algorithms for computations on
(quotients of) path algebras and their representations. The software ``pmgap'' uses those data structures to
represent equioriented commutative grids and persistence modules over them,
and implements the algorithms of this paper not in QPA.

\paragraph{Randomly generated persistence modules}
For the computational experiments below, given values for $n$ and $d$ we randomly generate persistence modules $V$ (with $\mathbb{F}_2$ coefficients)
over the commutative grid $\Gf{2,n}$, such that all the vector spaces of $V$ have dimension $d$.

Below, whenever we say to randomly generate a $j\times k$ $\mathbb{F}_2$-matrix $M$, we simply generate a matrix with entries independently and uniformly sampled from $\mathbb{F}_2$. If required, it is also possible to randomly choose a valid rank and then generate a random matrix with that rank. However, this comes at the cost of more computation time to generate the random matrices.

We use the following procedure to randomly generate the persistence module $V$.
First, we randomly generate $d\times d$ $\mathbb{F}_2$-matrices for each of the solid arrows below:
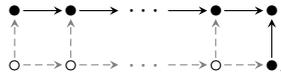
\[
  \newcommand{\bb}{\bullet}
  \newcommand{\rard}{\rar[dashed,gray]}
  \newcommand{\uard}{\uar[dashed,gray]}
  \begin{tikzcd}[graphstyle,every matrix/.append style={name=m}]
    \bb \rar & \bb \rar & \;\cdots\; \rar & \bb \rar & \bb \\
    \circ \rard \uard & \circ \rard \uard & {\color{gray}\;\cdots\;} \rard & \circ \rard\uard & \bb\mathrlap{.} \uar
  \end{tikzcd}
\]
Then, for each square from right to left, we iteratively compute pullbacks (to guarantee commutativity) and multiply with another random matrix (to reach the correct dimension $d$ and to add more randomness).
That is, given $d\times d$ matrices representing the linear maps $f$ and $g$ as below:
\[
  \begin{tikzcd}
    & \mathbb{F}_2^d \rar{f} & \mathbb{F}_2^d \\
    & \mathbb{F}_2^k \uar[dashed,swap]{\phi_1} \rar[dashed]{\phi_2} & \mathbb{F}_2^d \uar{g} \\
    \mathbb{F}_2^d \ar[ur,dashed]{}{\phi_3} \ar[ruu,dashed,bend left]{}{\phi_1\phi_3} \ar[rru,dashed,swap, bend right]{}{\phi_2\phi_3}&&
  \end{tikzcd},
\]
we compute
(matrices with respect to some basis of) the pullback maps $(\phi_1, \phi_2)$. Then, we randomly generate a $k \times d$ matrix representing $\phi_3$, and obtain the commutative diagram
\[
  \begin{tikzcd}
    \mathbb{F}_2^d \rar{f} & \mathbb{F}_2^d \\
    \mathbb{F}_2^d \uar[]{\phi_1\phi_3} \rar[swap]{\phi_2\phi_3} & \mathbb{F}_2^d\mathrlap{.} \uar{g}
  \end{tikzcd}
\]
It is clear that a persistence module $V$ over $\Gf{2,n}$ is obtained by the above.

\paragraph{Computational experiments}
We measure the time needed to compute the \EDIT{interval-decomposable replacement} using pmgap for some small values of $n$ and $d$. Computations were performed on Ubuntu~20.04.2~LTS running in WSL1 inside a Windows~10 Pro machine
with an AMD Ryzen 5 5600X 6-Core\footnote{Note that the current implementation does not take advantage of multiple cores or threads.} Processor.
In Table~\ref{table:runtimes1}, we display the resulting runtimes in milliseconds. Each timing (each entry in the table) is measured as the average of at least five runs. Each run consists of the computation of the compressed multiplicity and \EDIT{interval-decomposable replacement} of a given persistence module. Additional runs are performed as needed so that the total time taken exceeds $100~\text{ms}$, to ensure reliable measurement of runtimes; this is only needed for the smaller values of $n$ and $d$. Note that we exclude the time taken for generating the underlying path algebra, list of interval representations, and the persistence modules $V$.

\begin{table}[!h]\centering
  \caption{Runtimes (in ms) for the \EDIT{interval-decomposable replacement} using pmgap}
  \begin{threeparttable}
    \begin{tabular}{l|rrrrrr}
      \toprule
      {\diagbox{$n$}{$d$}} &     100 &     200 &     400 &      800 \\
      \midrule
      4  &   11.88 &   34.40 &   112.40 &   471.80 \\
      8  &  131.20 &  328.20 &  1,152.80 &  5,115.60 \\
      16 & 1,881.40 & 4,415.80 & 14,918.80 & 67,171.60 \\
      \bottomrule
    \end{tabular}
    \begin{tablenotes}\footnotesize
    \item[*] Runtimes are measured as an average of at least five runs.
    \item[*] Runtimes do not include time needed for generating the underlying path algebra,
      list of interval representations, and the persistence modules.
    \end{tablenotes}
  \end{threeparttable}
  \label{table:runtimes1}
\end{table}

For completeness, we also time the following operations:
generation of the underlying path algebra and its list of interval representations,
generation of a random  persistence module,
computation of the \EDIT{interval-decomposable replacement}.
The results (of just one run each) are displayed in Table~\ref{table:runtimes2}.
Note that the underlying path algebra and its list of interval representations do not depend on the dimension $d$.
Thus, we time that operation only once for each $n$.

\begin{table}[!h]\centering
  \caption{Runtimes (in ms) in pmgap}
  \begin{threeparttable}
    \begin{tabular}{cr|rrrr}
      \toprule
      \textbf{operation} & \boldmath$n$ &         &         &         &         \\
      \midrule
      \multirow{3}{*}{\begin{tabular}{c}algebra\\and\\its intervals\end{tabular}}
      & 4  & & & & \multicolumn{1}{r}{31.0}\\
      & 8  & & & & \multicolumn{1}{r}{562.0}   \\
      & 16 & & & & \multicolumn{1}{r}{16,578.0} \\
      \bottomrule\\
      \toprule
                & \boldmath$d$ &   100.0 &   200.0 &   400.0 &   800.0  \\
      \textbf{operation} & \boldmath$n$ &         &         &         &         \\
      \midrule
      \multirow{3}{*}{\begin{tabular}{c} random\\persistence\\module\end{tabular}}
      & 4  &    46.0 &   171.0 &   640.0 &  2594.0 \\
      & 8  &    94.0 &   375.0 &  1,422.0 &  5,516.0 \\
      & 16 &   219.0 &   781.0 &  3,047.0 & 12,609.0 \\
      \midrule
      \multirow{3}{*}{\begin{tabular}{c} interval-\\\EDIT{decomposable}\\\EDIT{replacement}\end{tabular}}
      & 4  &    15.0 &    31.0 &   109.0 &   484.0 \\
      & 8  &   125.0 &   328.0 &  1,156.0 &  5,141.0 \\
      & 16 &  1,828.0 &  4,422.0 & 14,953.0 & 67,218.0 \\
      \bottomrule
    \end{tabular}
  \end{threeparttable}
  \label{table:runtimes2}
\end{table}

\paragraph{Demonstrations}
The pmgap repository \cite{pmgap} contains demonstrations for these computations.

We also provide a browser-based implementation \cite{intervaldemo} demonstrating the computation of
\EDIT{interval-decomposable replacement} of randomly generated persistence modules.
Note that the browser-based demo \cite{intervaldemo} was developed separately of pmgap,
and does not rely on the installation of pmgap.


\section*{Acknowledgements}
  On behalf of all authors, the corresponding author states that there is no conflict of interest.

\bibliographystyle{alpha}
\bibliography{refs}

\end{document}